\documentclass[leqno, final]{siamltex}

\usepackage{amsmath,  amssymb, enumerate, xcolor, color, bm}
\usepackage{ graphicx, float, wasysym, mathrsfs}
\usepackage{stmaryrd}
\usepackage{xfrac, bbm}
\usepackage{booktabs, cite}
\usepackage{subcaption}
\usepackage{mathtools}
\usepackage{caption}

\usepackage{aicescover}

\def\ds{\displaystyle}

\newcommand{\Ccinf}{\mathcal{C}_c^\infty}
\newcommand{\Lp}[1]{{L_{#1}}}
\newcommand{\Ltwo}{\Lp{2}}
\newcommand{\Linf}{\Lp{\infty}}
\newcommand{\Wmp}[2]{W^{#1}_{#2}} 
\newcommand{\Hm}[1]{H^#1} 
\newcommand{\Hone}{\Hm{1}} 
\newcommand{\normgen}[1]{\Vert #1 \Vert}	
\newcommand{\Hnorm}[2]{\normgen{#1}_{\Hone(#2)}}	
\newcommand{\Ltwonorm}[2]{\normgen{#1}_{\Ltwo({#2})}}	
\newcommand{\Linfnorm}[2]{\normgen{#1}_{\Lp{\infty}({#2})}}	
\newcommand{\Winfnorm}[2]{\normgen{#1}_{\Wmp{1}{\infty}({#2})}}	

\newcommand{\mbf}[1]{{\bm{{#1}}}}
\newcommand{\uu}{\mbf{u}}
\newcommand{\vv}{\mbf{v}}
\newcommand{\ww}{\mbf{w}}	
\newcommand{\ff}{\mbf{f}}	
\newcommand{\phiv}{{\bm{\varphi}}}	

\newcommand{\uh}{\uu^h}
\newcommand{\vh}{\vv^h}
\newcommand{\wh}{\ww^h}
\newcommand{\phivh}{\phiv^h}

\newcommand{\df}[1]{\,\mathrm{d}#1}
\newcommand{\dx}{\df{x}}
\newcommand{\dss}{\df{s}}
\newcommand{\dt}{\df{t}}

\newcommand{\mv}[2]{\langle{#1},{#2}\rangle_E}
\newcommand{\inprod}[2]{\langle {#1}, {#2} \rangle}

\newcommand{\fl}{\hat{\ff}(\vh)}
\newcommand{\flent}{\ff^\star(\vh)}
\newcommand{\diff}{D}

\newcommand{\el}{\kappa}
\newcommand{\sumk}{\sum_{\el, n}}
\newcommand{\defeq}{\vcentcolon =}

\newcommand{\gradb}{\overline{\nabla \vv}_\el}
\newcommand{\Res}{\mathrm{Res}}
\newcommand{\Resb}{	\overline{\Res}_{\el}}
\newcommand{\BResb}{\overline{\mathrm{BRes}}_{\el}}
\newcommand{\BDG}{\mathcal{B}_{DG}}
\newcommand{\BSC}{\mathcal{B}_{SC}}
\newcommand{\epsk}{\varepsilon_\el}
\newcommand{\dom}{\Omega}
\newcommand{\domt}{\dom_T}

\newcommand{\jump}[1]{\llbracket #1 \rrbracket^+_-}
\newcommand{\proj}{\Pi}
\newcommand{\err}[1]{e_\proj( #1)}
\newcommand{\cnsserr}{e^\proj_{\vh}}

\newcommand{\findimn}{\mathcal{V}_n^q}
\newcommand{\findim}{\mathcal{V}^q}

\title{On the convergence of a shock capturing discontinuous Galerkin method for nonlinear hyperbolic systems of  conservation laws }

\author{Mohammad Zakerzadeh \footnotemark[2]
        \and Georg May \footnotemark[2] }

\begin{document}
\aicescoverpage
\maketitle
\renewcommand{\thefootnote}{\fnsymbol{footnote}}
\footnotetext[2]{Aachen Institute for Advanced Study in Computational Engineering Science (AICES), 
RWTH Aachen, 52062, Aachen, Germany
 ({\tt \{zakerzadeh, may\}@aices.rwth-aachen.de}).}

\begin{abstract}
In this paper, we present a  shock capturing discontinuous Galerkin (SC-DG) method for nonlinear systems of conservation laws in several space dimensions and analyze its stability and convergence. The scheme is realized as a  space-time formulation in terms of entropy variables using an entropy stable numerical flux. While being similar to the method proposed in \cite{hiltebrand2014entropy}, our approach is new in that we do not use streamline diffusion (SD) stabilization.  It is proved that  an artificial viscosity-based nonlinear shock capturing mechanism is sufficient  to ensure  both entropy stability and  entropy consistency, and consequently we establish convergence to an entropy measure-valued (emv) solution. 
 The result is valid for general systems and arbitrary order discontinuous Galerkin method. 
\end{abstract}



\pagestyle{myheadings}
\thispagestyle{plain}
\markboth{M. ZAKERZADEH AND G. MAY}{Convergence of SC-DG method for hyperbolic systems}

\section{Introduction}
The class of nonlinear systems of conservation laws contains many important examples, such as the  Euler equations and the  Navier-Stokes equations. 
The general form of a nonlinear $m$-system of conservation laws in several space dimensions is
\begin{equation}\label{Eq::CL}
\begin{cases} \uu_{t}+\sum^{d}\limits_{k=1} \ff^k(\uu)_{x_k} = 0, \\
\uu(x, 0) = \uu_0(x),
\end{cases}  
\end{equation}
where the unknowns $\uu = \uu(x,t) \colon \mathbb{R}^d \times {[0, \infty)} \to \mathbb{R}^m$ are the \emph{conserved variable}s and  $\ff^k\colon \mathbb{R}^m \to \mathbb{R}^m, \, k=1, \dotsc, d$ are (nonlinear) smooth \emph{flux function}s with $d = 1,2, 3$. 

The initial condition $\uu_0(x)$ is assumed to have compact support to avoid technicalities arising from boundary conditions. Using this assumption together with finite speed of propagation in hyperbolic problems, one may assume that the solution $\uu(x,t)$ has compact support for any finite time $t$ and vanishes for $\vert x \vert $ large. 

It is well-known that \eqref{Eq::CL} can produce shocks and discontinuities in finite time; hence the solution cannot be interpreted in the classical sense. This motivates one to introduce the concept of \emph{weak solution} which is defined as a bounded function $\uu$ that satisfies \eqref{Eq::CL} in distributional sense, i.e.\ 
\begin{equation}\label{Eq::weaksol}
\int_0^\infty\! \int_{\mathbb{R}^d} \! 
\inprod{\uu}{\phiv_t}
 + \sum_{k=1}^d \inprod{\ff^k(\uu)}{  \phiv_{x_k}}  \dx \dt + \int_{\mathbb{R}^d} \! \inprod{\uu_0(x)}{\phiv(x, 0)} \dx = 0,
\end{equation}
for all functions $ \phiv \in (\Ccinf(\mathbb{R}^d \times [0, \infty))\,)^m$.
Here the notation  $\inprod{\uu}{\ww}$  denotes the inner product between vectors $\uu$ and $\vv$ in the state space $\mathbb{R}^m$. Also we will use the notation $\mbf{a} \cdot \mbf{b}$ as the notation for the inner product of vectors $\mbf{a}$ and $\mbf{b}$ in the physical space $\mathbb{R}^d$. 

In order to single out the physically admissible solutions we require solution $\uu$ to satisfy
the \emph{entropy inequality condition} 
\begin{equation}\label{Eq::ent-cond}
U(\uu)_t + \sum_{k=1}^d F^k(\uu)_{x_k} \leq 0,
\end{equation}
in the distributional sense for all entropy flux $U(\uu)\colon \mathbb{R}^m \to \mathbb{R}$ and associated entropy flux functions $F^k(\uu) \colon \mathbb{R}^m \to \nobreak \mathbb{R}$ for $k=1, \dots, d$.  Here $U$ is convex and $(U, F)$ satisfy the compatibility condition
$\partial_{\uu} F^k(\uu) = \partial_\uu U(\uu) \partial_\uu \ff^k(\uu)$. 
By defining \emph{entropy variables} as $\vv = (U_\uu)^T$ one can recast \eqref{Eq::CL} in symmetric form as 
\begin{equation*}
\mbf{u_v}\vv_t + \sum_{k=1}^{d} \ff^k_{\vv} \vv_{x_k} = 0,
\end{equation*} 
such that the matrix $\uu_\vv$ is symmetric positive definite and the matrices $\ff_{\vv}^k$ are symmetric. 

In general, the best a priori estimate one can get for the solutions of \eqref{Eq::CL} is the so-called \emph{entropy stability}. This originates from  the entropy inequality condition \eqref{Eq::ent-cond} by  integrating it over the spatial domain and considering an arbitrary time $T$  combined with compact support assumption which lead to the following \emph{global entropy inequality}
\begin{equation}\label{Eq::entr-glob-semi-dis-ineq}
\frac{\mathrm{d}}{\dt} \int_{\mathbb{R}^d} \! U(\uu) \dx \leq 0 \Longrightarrow \int_{\mathbb{R}^d} \! U(\uu(x, T)) \dx \leq \int_{\mathbb{R}^d} \! U(\uu(x,0)) \dx.
\end{equation}
This property can be viewed as the nonlinear extension of $\Ltwo$ stability for systems of conservation laws and is desirable to be kept for the approximate solution $\uh$ as well.  This is the motivation behind  \emph{entropy stable schemes}, which were originally  introduced  by Tadmor~\cite{tadmor1987numerical}. 
In a  finite volume framework, these methods  have been extended to higher order Essentially Non-Oscillatory (ENO) schemes very recently \cite{ulrikphd,fjordholm2012arbitrarily}.	
In the finite element context, in \cite{hughes1986new} entropy stability is constructed by adding streamline diffusion (SD) in space-time formulation. Later formulations with streamline diffusion and with/without shock capturing (SC) term are introduced in \cite{ johnson1987convergence,johnson1990convergence,  szepessy1989convergence,szepessy1991convergence}. The extension to DG methods is presented in \cite{jaffre1995convergence}.

The above-mentioned methods are designed to satisfy the entropy stability condition; however this is not sufficient to conclude any sort of convergence for the numerical scheme in the general case due to lack of enough a priori information on the solution. Trying to obtain some sort of convergence leads to  an even weaker  notion of  solution,  the so-called \emph{entropy measure-valued}  (emv) solutions. These types of solutions, introduced by DiPerna~\cite{diperna1985measure}, are more general than weak solutions and permit a meaningful convergence theory for numerical schemes approximating \eqref{Eq::CL}. We discuss  this concept later in \S \ref{Sec::mv}.

For scalar equations, the emv solution  contains the entropy weak solution as a special case (when the initial data is a Dirac measure, see DiPerna \cite{diperna1985measure}). Using this theory,  convergence to entropy weak solutions of scalar conservation laws  has  been  established for both continuous and discontinuous streamline diffusion finite element methods \cite{jaffre1995convergence, johnson1987convergence, johnson1990convergence, szepessy1989convergence, szepessy1991convergence}.
In  the case of systems, convergence to an  emv solution has been proved very recently in \cite{ulrikphd} for TeCNO schemes in the finite volume context and in \cite{hiltebrand2014entropy} for an SCSD discontinuous Galerkin (SCSD-DG) method. 

On the other hand, despite the apparent need to include SD terms to control the residual in these schemes,  ideas questioning  the necessity and even adequacy of linear stabilization (e.g.\, streamline diffusion) have gained momentum \cite{nazarov2013convergence,ern2013weighting, guermond2010entropy, guermond2008entropy}. Furthermore, while SD stabilization is often included in the analysis of DG schemes, it is not commonly found in practical implementations. (There is a plethora of examples, e.g., \cite{hartmann2006adaptive, persson2006sub, burgess2012hp}.) Recently, Nazarov in \cite{nazarov2013convergence} suggested a stripped-down version of the SCSD continuous Galerkin method of \cite{johnson1990convergence} for scalar equations using linear (continuous) finite elements. The formulation of \cite{nazarov2013convergence} disregards the SD term and utilizes a residual based shock capturing as the only stabilization mechanism while it is proved that the approximate solution still  converges to the entropy weak solution.

In the present paper we propose a class of DG schemes for~\eqref{Eq::CL}, using only a suitable nonlinear shock-capturing term for stabilization. We will show that our method is entropy stable and satisfies the global  entropy inequality~\eqref{Eq::entr-glob-semi-dis-ineq}. The main goal of this paper is to prove that
uniform $\Linf$ bounded solutions of the suggested scheme converge to an entropy measure-valued solution of~\eqref{Eq::CL} for  arbitrary (fixed) order of polynomial approximation.  

The framework presented in \cite{hiltebrand2014entropy}, where convergence of a SCSD-DG method was proved, is the skeleton of this work.  In the present paper we extend the result of \cite{hiltebrand2014entropy} not only by proving that we can obtain adequate residual control without using  streamline-diffusion stabilization, but we also use refined estimates, resulting in a shock capturing operator using nonlinear viscosity that is higher order small compared to  \cite{hiltebrand2014entropy}. This results in a  less diffusive method.
 
Section \ref{Sec::mv} gives a brief review  
on Young measures and mv solutions which will later be used in the convergence proof. 
The space-time DG  framework is introduced in \S \ref{Chap::DGform}. This section also includes the explicit forms of the numerical diffusion and shock capturing operators. 
In \S \ref{chap::stab} the fully discrete entropy inequality  and a BV-estimate are obtained and  \S\ref{sect::conv} includes the proof of convergence to an entropy measure-valued solution. Furthermore, in \S\ref{sec::numeric} we provide some numerical examples to show the applicability of the method. Appendix \ref{Sec::App} contains the proof of Lemmas \ref{Lem-app-ineq} and \ref{Lem-app-bnd-ineq}.

\section{Entropy measure-valued solutions}
\label{Sec::mv}
The notion of measure-valued solution is a generalization of the standard  distributional (weak) solution of \eqref{Eq::CL}.
We follow \cite{diperna1985measure} and define a measure-valued solution of \eqref{Eq::CL} as  a measurable map $\bm{\mu}$ from the physical domain $\mathbb{R}^d \times \mathbb{R}_+$ to the space of non-negative measures  with unit mass over the state domain $\mathbb{R}^m$,
\begin{equation*}
\bm{\mu} \colon y=(x,t) \in (\mathbb{R}^d \times \mathbb{R}_+) \mapsto \bm{\mu}_y \in \mathrm{Prob}(\mathbb{R}^m),
\end{equation*}
which satisfies \eqref{Eq::CL} in the following sense
\begin{equation} \label{Eq-mvsol}
\int_{\mathbb{R}^d} \! \int_{\mathbb{R}_+}\! \inprod{
\mv{\bm{\sigma}}{\bm{\mu}_y}}{\phiv_t}  + \sum^d_{k=1} \inprod{
\mv{\ff^k(\bm{\sigma})}{\bm{\mu}_y}}{ \phiv_{x_k}} \dx \dt = 0,
\end{equation}
for all test functions $\phiv \in (\Ccinf(\mathbb{R}^d \times \mathbb{R}_+))^m$. Here $y$ and $\bm{\sigma}$ denote the generic variables in space-time domain $\mathbb{R}^d \times \mathbb{R}_+$, and state domain $\mathbb{R}^m$, respectively. Moreover, the notation $ \mv{\mbf{g}(\bm{\sigma})}{\bm{\mu}_y}$ denotes the expectation of function $\bm{g}$ with respect to the probability measure $\bm{\mu}_y$ as
\begin{equation*}
\mv{\bm{\mu}_y}{\mbf{g}(\bm{\sigma})} \defeq \int_{\mathbb{R}^m} \! \mbf{g}(\bm{\sigma})\df{\bm{\mu}_y}, \quad \mbf{g}\colon \mathbb{R}^m \rightarrow \mathbb{R}^m.
\end{equation*}
Since the system \eqref{Eq::CL} has an entropy extension with entropy pair $(U, F)$, $\bm{\mu}$ is called an admissible (or entropy) measure-valued solution if
\begin{equation}\label{Eq-mv-entrpy}
\int_{\mathbb{R}^d}\!\int_{\mathbb{R}_+}\! \varphi_t \mv{U(\bm{\sigma})}{\bm{\mu}_y} + \sum^d_{k=1} 
\varphi_{x_k} \mv{F^k(\bm{\sigma})}{\bm{\mu}_y}   \dx \dt \geq 0,
\end{equation}
for all $0 \leq \varphi \in \Ccinf(\mathbb{R}^d \times \mathbb{R}_+)$.  
 The linearity of \eqref{Eq-mvsol} and \eqref{Eq-mv-entrpy} with respect to $\mbf{\mu}$
 helps  prove convergence of a bounded sequence of solutions produced by a vanishing viscosity method, which is a significant problem for traditional weak solutions to nonlinear systems. 
The following \emph{Young's theorem} provides such an appropriate interpretation of convergence:
\begin{theorem}[Theorem 2.1 of ~\cite{szepessy1989convergence}] \label{Thm-mvsol}
Let $\uu_j$ be a uniformly bounded sequence in \mbox{$\Linf(\mathbb{R}^d \times \mathbb{R}_+)$}, i.e., for some constant $C$,
\begin{equation*}
\Linfnorm{\uu_j}{\mathbb{R}^d \times \mathbb{R}_+}  \leq C, \qquad j= 1, 2, 3, \dotsc . 
\end{equation*}
Then there exists a subsequence (again denoted) $\uu_j$ and a family of measurable probability measures $\bm{\mu}_y \in \mathrm{Prob}(\mathbb{R}^m)$, such that $\supp \mbf{\mu}_y$ is contained in $\{ y \in \mathbb{R}^d \! \times \mathbb{R}_+, \, \vert y \vert \leq C \}$ and the $\Linf$ weak-$*\,$limit,
\begin{equation*}\label{Eq::T-wk*cnv}
\mbf{g}(\uu_j(\cdot)) \overset{*}{\rightharpoonup} \bar{\mbf{g}} (\cdot), 
\end{equation*}
exists for all continuous functions $\mbf{g}$ and for almost all points $y \in \mathbb{R}^d \times \mathbb{R}_+$, where 
 $\bar{\mbf{g}}   \defeq \nobreak  \mv{\bm{\mu}_y}{\mbf{g}(\bm{\sigma})}$.
\end{theorem}

\section{Space-time SC-DG formulation}
\label{Chap::DGform}
Here, we introduce the shock capturing discontinuous Galerkin (SC-DG) method for nonlinear systems of conservation laws  \eqref{Eq::CL}. 
A space-time framework, similar to that used in \cite{hiltebrand2014entropy, johnson1987convergence, johnson1990convergence, szepessy1989convergence, szepessy1991convergence},  is proposed for discretization of the problem. We introduce the  space-time triangulation, the approximation space and in particular the structure of the shock capturing term. 

\subsection{Space-time triangulation} 
\label{sec::spacetime}
Adopting the compact support assumption for the solution in a finite time interval ${[0, T]}$, we consider the space domain $\dom \subset \mathbb{R}^d$ such that $\mathrm{supp}\,\uu(\cdot, t) \subset \dom$ at each time $t \in {[0, T]}$.
In order to discretize \eqref{Eq::CL}, let $0= t_0 < t_1 < ... < t_N = T$ be a sequence representing  discrete time steps, and let $I_n = {[t_n, t_{n+1})}$ be the corresponding time intervals. We also denote the space-time domain and space-time slabs by $\domt \defeq \dom \times [0, T] $ and $S_n \defeq \dom \times I_n$, respectively. Moreover $d' \defeq \dim(\domt)$ represents the space-time dimension and clearly $d' = d+1$.

We consider a subdivision $\mathcal{T}_n = \{\el\}$  of $S_n$ into disjoint convex\footnote{The necessity of the convexity requirement becomes clear in the approximation estimates of the $\Hone$-projection \eqref{Eq-H1prj}. } finite elements.   
Without loss of generality, let us assume that 
\begin{equation*}\label{Eq::h_def}
h = \sup_{\el,n} h_\el < \infty, \qquad \el \in \mathcal{T}_n, n \in \{0, \ldots, N-1 \},
\end{equation*}
where $h_\el$ is the exterior diameter of a space-time cell $\el$.
The interior diameter of an element (the diameter of the inscribed circle) is denoted by $\rho_\el$.
We assume the following quasi-uniformity condition 
\begin{equation}\label{Eq-uniformity}
\frac{h}{\rho_\el} \leq \sigma, \qquad  \forall \el \in \mathcal{T}_n,
\end{equation}
with $\sigma > 0$ independent of $h$.   
The perimeter of $\el$ is defined by $p_\el = \Sigma_{e \in \partial \el} \vert e \vert$, where $\vert e \vert$ is the $d$-measure of the face. The uniformity assumption \eqref{Eq-uniformity} implies that (cf.\ \cite{cockburn1994error})
\begin{equation} \label{Eq-area-volume}
\frac{1}{\mu} \leq \frac{p_\el h_\el}{\vert \el \vert} \leq \mu,  \qquad \forall \el \in \mathcal{T}_n,
\end{equation}
for some $\mu > 0$ independent of $h$. Typically, $\el$ might be a tetrahedron or a prism defined as $K \times I_n$, where $K$ corresponds to spatial triangulation on $\mathbb{R}^d$. Seeking easier notation, from now on we present our formulation for prisms. Note, however, that there is no restriction to extend this framework to tetrahedra or tilted prisms (cf. \cite{jaffre1995convergence} for more discussion). 

Temporal trace values are denoted by $ \wh_{n,\pm}(x) \defeq \wh(x, t^n_\pm)$
and to define spatial trace quantities, if $\mbf{n}$ is the outward normal to the spatial interface $\partial K$, we set $\ww_{K, \pm}(x,t) \defeq \lim_{\epsilon \to 0 } \ww(x \pm \epsilon \mbf{n}, t) $ as the associated trace values on an interface.
Also we introduce the notation $\jump{ \ww } \defeq \ww_+  -  \ww_- $ for the (spatial or temporal) jump values on the cell interface.

\subsection{Variational formulation}
The finite dimensional space for the approximate solution is defined as
\begin{equation*}\label{Eq::variational}
\findimn =\{ \ww \in (\Ltwo(S_n))^m \colon \ww \vert_\el \in \left( \mathbb{P}_q(\el)\right)^m, \forall \el \in \mathcal{T}_n\},\qquad n=0, \ldots, N-1,
\end{equation*} 
where $\mathbb{P}_q(\el)$ is the space of polynomials of at most degree $q$ on a domain $\el \subset \mathbb{R}^{d'}$. 
 We also denote $\findim = \prod_{n=0}^{N-1} \findimn$ as the approximation space in global space-time domain. 
The approximating functions are considered discontinuous both in space and time.
 
The proposed shock capturing discontinuous Galerkin  method has the following quasi-linear (nonlinear in first argument and linear in the second one) variational form in terms of entropy variables: Find $\vh \in \findim $ such that
\begin{equation}\label{Eq-DGSC}
\mathcal{B}(\vh, \wh) = \BDG(\vh,\, \wh) + \BSC(\vh, \wh) = 0, \qquad \forall \wh \in \findim. 
\end{equation}
Note that  we realize the functions in terms of entropy variables $\vh$ which are the basic unknowns and the dependent conservative variables are derived via mapping $\uu(\vh)$. In our notation, this mapping is sometimes  omitted, e.g., $\ff(\vh)$ is written rather than $\ff (\uu(\vh ))$.

The scheme \eqref{Eq-DGSC} can be seen as the stripped-down version of the method suggested in \cite{hiltebrand2014entropy}, by disregarding the streamline diffusion (SD) term which is usually added to control the residual.
 
In the following we explain the details and explicit form of terms in \eqref{Eq-DGSC}.

\subsection{DG quasi-linear form}
Using the test function $\wh \in  \findim$ to penalize the interior residual of the cell, jumps of temporal values and spatial flux and applying the integration by part leads to 
\begin{align}\label{Eq::DG-quasilinear}
\BDG(\vh, \wh) =  &- \sumk \int_{I_{n}} \int_{K} \inprod{ \uu(\vh)}{\wh_{t}} + \sum^{d}_{k=1} \inprod{\ff^k(\vh)}{ \wh_{x_{k}}} \dx \dt  \nonumber \\  
 & +\sumk \int_{K}  \inprod{\uu(\vh_{n+1,-})}{\wh_{n+1, -}} - \inprod{\uu(\vh_{n,-})}{\wh_{n, +}}  \dx \nonumber \\
  & + \sumk \int_{I_n} \int_{\partial K} \inprod{ \fl}{ \wh_{K, -}} \dss \dt.
\end{align}
 Also we assume that the initial data $\vh_{0,-} = \vh(x,0_- )$ is obtained from a suitable projection (e.g. $\Ltwo$-projection or the proposed $\Hone$-projection in \eqref{Eq-H1prj}) of the initial data $\vv_0(x) = \vv(\uu_0(x))$. 

Here, $\fl \defeq \hat{\ff}(\vh_{K, -}, \vh_{K, +}; \mbf{n})$
denotes the (spatial) numerical flux
function, a vector-valued function of two interface states $\vh_{K, \pm}$ and the interface normal $\mbf{n}$,
which  is considered to be consistent and conservative. 
Also this numerical flux is supposed to be entropy stable, i.e.\
following \cite{hiltebrand2014entropy}, we consider the spatial numerical flux in  the viscosity  form as
\begin{equation} \label{Eq::diff_op}
\fl = \flent  - \dfrac{1}{2} \diff(\vh) \big( \vh_{K, +} - \vh_{K, -} \big),
\end{equation}
where $\flent \defeq \ff^\star(\vh_{K, -}, \vh_{K, +}; \mbf{n})$ denotes the \emph{entropy conservative flux} and $\break \diff(\vh) \defeq \diff(\vh_{K, -}, \vh_{K, +}; \mbf{n} )$ is the required  numerical diffusion matrix to obtain the entropy stability. 

For comprehensive discussion on entropy conservative and entropy stable fluxes we refer to the seminal paper by Tadmor \cite{tadmor2003entropy}, and just mention that for a general system of conservation law $\flent$ can be written in the form
\begin{equation}\label{Eq::entrp-cnsrv-flux}
\flent = \int^1_0 \ff(\vh(\theta))\cdot\mbf{n} \df{\theta}.
\end{equation}
where $\vh(\theta)$ is a straight line parameterization connecting the two states $\vh_{K,-}$ and  $\vh_{K, +}$ as 
\begin{equation}\label{Eq::barth-param}
\vh(\theta) = \vh_{-} + \theta  \jump{\vh}.
\end{equation}
 
Unfortunately, \eqref{Eq::entrp-cnsrv-flux} does not necessarily have a closed-form and is hard to calculate. We refer to \cite{tadmor2003entropy} for discussions on the practical method for obtaining entropy conservative flux. Also we refer to \cite{fjordholm2012arbitrarily} for explicit formulation of entropy conservative fluxes for Euler and shallow water equations.
    
Moreover, we set $\diff$  as a symmetric positive definite matrix with a uniform spectral bound, i.e.\ there exist positive constants $c$ and $C$ independent of $\vh$ such that 
\begin{equation}
0 < c \inprod{ \ww}{ \ww} \leq \inprod{ \ww}{ \diff(\vh) \ww} \leq C \inprod{ \ww}{ \ww}, \qquad \forall \ww  \neq 0. \label{Eq::D_bound}
\end{equation}

 In order to determine the diffusion operator explicitly we follow \cite{hiltebrand2014entropy} and define
\begin{equation*}
\diff(\vh_{K,-}, \vh_{K,+}; \mbf{n}) =
 \tilde{R}_{\mbf{n}} P(\Lambda_{\mbf{n}}) \tilde{R}^T_{\mbf{n}}.
\end{equation*}

Here, $\Lambda_{\mbf{n}}$ and $\tilde R_{\mbf{n}}$ are eigenvalue and (scaled) eigenvector matrices of the Jacobian matrix $(\ff\cdot \mbf{n})_\uu$ in the normal direction $\mbf{n}$, calculated at an averaged state between $ \vh_{K,-}$ and $\vh_{K,+} $ (e.g., Roe average or arithmetic average). The scaled matrix of right eigenvectors is given as  $\tilde{R}_\mbf{n} = R_\mbf{n} T$ such that $\tilde{R}_\mbf{n}\tilde{R}^T_\mbf{n} = \uu_\vv$. 
Here, matrix $P$ is a non-negative matrix that can be constructed as Roe-type or Rusanov-type \cite{fjordholm2012arbitrarily}:
\begin{itemize}
\item Roe-type diffusion operator
\begin{equation*}
P(\Lambda_{\mbf{n}}) = \mathrm{diag}(\vert \lambda_1 \vert, \dotsc,\vert \lambda_m \vert ),
\end{equation*}
\item Rusanov-type diffusion operator
\begin{equation*}
P(\Lambda_{\mbf{n}}) =  \max(\vert \lambda_1 \vert, \dotsc,\vert \lambda_m \vert )I_{m\times m},
\end{equation*}
\end{itemize}  
where $\lambda_1 , \dotsc, \lambda_m $ are the eigenvalues of $(\ff \cdot \mbf{n})_\uu$.

It is worth mentioning that, by $C$ (or $c$) we will denote a positive constant independent of $h$, not necessarily the same at each occurrence.
 
\subsection{Shock capturing operator}
\label{Sec:scop}
In order to stabilize the scheme in the presence of  discontinuities we need to add a form of  artificial viscosity. We expect this operator to add a significant stabilization effect close to discontinuities, while  only a little viscosity is added in  smooth regions.  In this formulation, the residual of the finite element solution is used as a sensor for presence of discontinuities.
 
Here we follow \cite{barth1999numerical, hiltebrand2014entropy} in introducing the shock capturing  operator as
\begin{align}\label{Eq-B_SC}
\BSC (\vh, \ww^h) = \sumk \int_{I_n} \int_{K} \epsk \Big( \inprod{ \ww^h_{t}}{ \tilde{\uu}_{\vv} \vh_{t}}  
 + \sum^{d}_{k=1} \inprod{ \ww_{x_k}}{ \tilde{\uu}_{\vv} \vv_{x_k} } \Big) \dx \dt,
\end{align}
where the viscosity $\epsk$ is defined as
\begin{equation}\label{Eq-D_SC}
\epsk = \frac{h^{\alpha_1} C_1^{SC} \Resb + h^{\alpha_2} C_2^{SC} \BResb } {\gradb + h^{\theta}     }.
\end{equation}
Here,  $C_1^{SC}$ and $C_2^{SC}$ are two positive constants and $h^\theta$ is added as the regularization parameter  with parameter $\theta$ such that 
\begin{equation}\label{Eq::tetabound}
\theta \geq \max \{ \frac{d'}{2} - \frac{\alpha_1}{2}, \frac{d'}{2} - \alpha_2 \}.
\end{equation}
Also the viscosity strength parameters $\alpha_1$ and $\alpha_2$ are chosen such that
\begin{equation}\label{Eq::alfabound}
\alpha_1 \in (0,2), \qquad \alpha_2 > 0.
\end{equation}
The rationale behind these choices for $\theta$, $\alpha_1$ and $\alpha_2$ are discussed later in \S\ref{sect::conv}. It should be noted that the scaling of the viscosity coefficient \eqref{Eq::alfabound} is less diffusive, compared to range $\alpha_1 \in (0, 1)$ and $\alpha_2 \in (0, 1/2)$ in~\cite{hiltebrand2014entropy}, due to refined estimates used in \S\ref{sect::conv}. 

We denote  the local residual and space-time gradient as
\begin{align}
\label{Eq::residual}
\Res = \uu(\vh)_t + \sum_{k=1}^{d} \ff^k(\vh)_{x_k}, \\
\nabla \vv = (\nabla_{t} \vv, \nabla_{x_1} \vv, \dots, \nabla_{x_d} \vv)^T
\end{align}
and we have the following definitions for
the weighted cell and boundary residuals, and the weighted gradient, respectively,
\begin{align}
\Resb^2 &\defeq \int_{I_n} \int_{K} \inprod{ \Res}{ \vv_\uu(\vh(x,t)) \Res } \dx \dt, \label{Eq::Res} \\
\BResb^2 &\defeq  \int_K \vert \jump{\uu(\vh_{n})}\vert^2 \dx \nonumber \\ & \qquad + \int_{I_n  } \!  \int_{\partial K} \! \Big(  \vert \flent - \ff(\vh_{K, -})\cdot\mbf{n} \vert ^2  + \vert \frac{1}{2} \diff(\vh) \jump{\vh_{K}} \vert^2 \Big) \dss \dt, \label{Eq-BRes} \\
\gradb^2 &\defeq \int_{I_n}\int_{K} \inprod{ \vh_t}{ \tilde{\uu}_\vv \vh_t}  + \sum^d_{k=1} \inprod{ \vh_{x_k}}{ \tilde{\uu}_\vv \vh_{x_k}} \dx \dt \label{Eq::Bgrad}.
\end{align}
Here by $\tilde{\uu}_{\vv}$  we denote $ \uu_{\vv}(\tilde{\vv}_{n, k})$, and $\tilde{\vv}_{n, K}$ is the cell average 
defined as
\begin{equation}
\tilde{\vv}_\el \defeq \frac{1}{\vert \el \vert} \int_{I_n} \int_K \! \vh(x,t) \dx \dt.
\end{equation}

Now the proposed SC-DG method  \eqref{Eq-DGSC}  is well-defined. The rest of the paper is basically devoted to the proofs of entropy stability and convergence to entropy  measure-valued solution for \eqref{Eq-DGSC}.

\section{Energy analysis}
\label{chap::stab}
We first note that the approximate solution of \eqref{Eq-DGSC} satisfies the global entropy inequality in the fully discrete sense. Then, by adopting some additional assumptions, we show a weak BV-estimate. 
\subsection{Entropy stability}
The entropy stability result is given as the following theorem: 
\begin{theorem}[Theorem 3.1 of \cite{hiltebrand2014entropy}] \label{Thm::entropy-stab} Consider the system of conservation laws \eqref{Eq::CL}, equipped with strictly convex entropy function $U$ and corresponding entropy flux functions $F^{k}, k = 1, \dotsc, d$. 
Furthermore, assume that the exact and approximate solutions have compact support inside the spatial domain
$\dom$. Then, the SC-DG	 scheme \eqref{Eq-DGSC} approximating \eqref{Eq::CL} has the following properties:
\begin{enumerate}[(i)]
\item The scheme \eqref{Eq-DGSC} is conservative in the following sense: If $\uh= \uu(\vh)$ is the approximate solution, then
\begin{equation*} \label{Eq-T-3.1-part1-conserv}
\int_{\dom} \uu(\vh(x, t^N_-)) \dx = \int_{\dom} \uu(\vh(x, t^0_-)) \dx. 
\end{equation*}
\item The scheme \eqref{Eq-DGSC} is entropy stable i.e., the approximate solution $\uh$ admits the following fully discrete global entropy bounds, 
\begin{equation*} \label{Eq-T-3.1-part2-stbl}
\int_{\dom} U(\uu^*(t^0_-)) \dx \leq \int_{\dom} U(\uu(\vh(x, t^N_-))) \dx \leq  \int_{\dom} U(\uu(\vh(x, t^0_-))) \dx,
\end{equation*}
where $\uu^*(t^0_-)$ is called the minimum total entropy state of the projected initial data  and is defined as
\begin{equation*}
\uu^*(t^0_-) = \frac{1}{\vert \dom \vert } \int_{\dom} \uu(\vh(x, t^0_-)) \dx.
\end{equation*}
\end{enumerate}
\end{theorem}
\begin{proof}(Sketch)
The proof of this theorem is not strongly dependent on the presence of streamline-diffusion stabilization, and is in fact very similar to the proof presented in \cite{hiltebrand2014entropy}. We give only a sketch here, mainly with the aim to introduce terms that  facilitate exposition of the material in the following. Consult \cite{mythesis} for a more detailed version of the proof.

First we note that the conservation property~\eqref{Eq-T-3.1-part1-conserv} follows immediately from choosing $\wh \equiv 1$ in~\eqref{Eq-DGSC}. The second assertion is obtained by  considering the following decomposition of \eqref{Eq-DGSC}
and inserting $\wh = \vh$ in it, to prove a series of inequalities:
\begin{align}
	\BSC(\vh, \vh) & \geq 0, \nonumber \\
	\BDG^{(s)} (\vh, \vh) & =  - \sumk \int_{I_{n}} \int_{K}  \sum^{d}_{k=1} \inprod{\ff^k(\vh)}{ \wh_{x_{k}}} \dx \dt \nonumber \\ & \quad+ \sumk \int_{I_n} \int_{\partial K} \inprod{ \fl}{ \wh_{K, -}} \dss \dt \geq 0, \nonumber\\
	\BDG^{(t)} (\vh, \vh) & = - \sumk \int_{I_{n}} \int_{K} \inprod{ \uu(\vh)}{\wh_{t}} \dx \dt\nonumber \\
	&\quad  +\sumk \int_{K} \left( \inprod{\uu(\vh_{n+1,-})}{\wh_{n+1, -}} - \inprod{\uu(\vh_{n,-})}{\wh_{n, +}} \right) \dx \nonumber \\
	& \geq  \int_\dom U(\uu(\vh(x,t^N_-)) \dx - \int_\dom U(\uu(\vh(x,t^0_-)) \dx. \nonumber
\end{align}
These estimates together give the upper bound in~\eqref{Eq-T-3.1-part2-stbl}. The lower bound is obtained exactly as in~\cite{hiltebrand2014entropy}. 
\end{proof}

Now, assume that the spectral bound~\eqref{Eq::D_bound} holds, and  there exist some constants $c$ and $C$ independent of $\vh$, such that 
\begin{align} 
0 < c \inprod{ \ww}{ \ww} \leq \inprod{ \ww}{ \uu_\vv (\vh(x,t)) \ww} \leq C \inprod{\ww}{ \ww}, \qquad \forall \ww  \neq 0.  \label{Eq::u_v_bound}
\end{align}
Then, we can make the inequalities of the proof of Theorem \ref{Thm::entropy-stab} sharper (cf. \cite{hiltebrand2014entropy, mythesis} for more details)
\begin{align*}
	\BSC(\vh, \vh) &  \geq C   \sumk  \epsk \Ltwonorm{ \nabla \vh }{\el}^{2}, \nonumber \\
	\BDG^{(s)} (\vh, \vh) & \geq  C  \sumk \int_{I_n} \! \int_{\partial K} \! \vert \jump{\vh_K} \vert^2 \dss \dt, \nonumber \\ 
	\BDG^{(t)} (\vh, \vh) & \geq   C  \sumk \! \int_K \! \vert \jump{\vh_n} \vert^2 \dx  +  \int_\dom \!  U(\vh(x,t^N_-)) \dx - \!  \int_\dom \!  U(\vh(x,t^0_-)) \dx. 
\end{align*}
The global entropy inequality~\eqref{Eq-T-3.1-part2-stbl} together with the above inequalities  imply
\begin{align}\label{Eq::temp-corr}
  \sumk  \epsk \Ltwonorm{ \nabla \vh}{\el}^{2}  +  \sumk  \int_K \! \vert  \jump{\vh_{n}} \vert^2 \dx  &+  \sumk \int_{I_n} \!  \int_{\partial K} \! \! \vert  \jump{\vh_K} \vert^2 \dss \dt \leq  C( \vh_{0,-}),
\end{align}
which readily gives
\begin{equation}\label{Eq::visc_bound}
  \sumk  \epsk \Ltwonorm{ \nabla \vh}{\el}^{2}  \leq C.
\end{equation}
This result will be used in the later proofs. Also note that due to the conditions \eqref{Eq::u_v_bound} and \eqref{Eq::D_bound}, the weighted residual (in \eqref{Eq::Res}) and weighted gradient (in \eqref{Eq::Bgrad}) are  norms equivalent to  $ \Ltwonorm{ \Res}{\el} $ and $ \Ltwonorm{ \nabla \vh}{\el}$, respectively.

\subsection{BV-estimate}
In order to prove  convergence, we require a BV-estimate for the approximate solutions of the SC-DG method \eqref{Eq-DGSC}.  Before reaching to that point we need to state  Lemmas \ref{Lem-app-ineq} and \ref{Lem-app-bnd-ineq}. The proofs will be presented in the appendix:
\begin{lemma}\label{Lem-app-ineq} Let us assume that \eqref{Eq::u_v_bound} holds and there exists a
 uniform spectral upper bound for  $\ff_\vv$, i.e.\
\begin{equation}\label{Eq::f_v_bound}
 \inprod{ \ww}{\ff_\vv \ww} \leq C \inprod{ \ww}{  \ww}, \qquad \forall \ww \neq 0,
\end{equation}
where $C$ is uniform and independent of $\ww$.
 Then one can find a uniform upper bound with respect to $h$ for
\begin{enumerate}[(i)]
\item $\ds h^\gamma \sumk \Resb$ if $\gamma \geq \dfrac{d'+\alpha_1}{2}$.
\item $\ds h^\gamma \sumk \Resb \Ltwonorm{ \nabla \vh}{\el} $ if 
$\gamma \geq \alpha_1 $.
\end{enumerate}
Moreover, these expressions vanish as $h \to 0$, if the inequalities hold strictly. 
\end{lemma}

A similar lemma can be stated for the boundary residual terms:  
\begin{lemma}\label{Lem-app-bnd-ineq} Assuming that \eqref{Eq::D_bound} and \eqref{Eq::f_v_bound} hold,  one can find a uniform upper bound with respect to $h$ for
\begin{enumerate}[(i)]
\item $\ds h^\gamma \sumk \BResb$ if $\gamma \geq \dfrac{d'}{2}$.
\item $\ds h^\gamma \sumk \BResb \Ltwonorm{ \nabla \vh }{\el} $ if 
$\gamma \geq \alpha_2 $.
\end{enumerate}
Moreover, these expressions vanish as $h \to 0$, if the inequalities hold strictly. 
\end{lemma}

Now the BV-estimate is obtained as a corollary of Theorem \ref{Thm::entropy-stab}, starting from inequality~\eqref{Eq::temp-corr}: 
\begin{corollary} Let the assumptions of Theorem \ref{Thm::entropy-stab} hold and  the diffusion matrix $\diff(\vh)$ be spectrally bounded as in \eqref{Eq::D_bound}. Also we assume similar spectral boundedness for $\uu_\vv$, i.e.,~\eqref{Eq::u_v_bound} holds. Then the approximate solution $\vh$ satisfies the following weak BV-estimate:
\begin{align} \label{Eq-T-3.1-part3-BV}
& \sumk \int_K \vert \jump{\vh_{n}} \vert^2 \dx 
 +  \sumk \int_{I_n} \int_{\partial K} \vert \jump{\vh_{K}}\vert^2 \dss \dt \nonumber \\
+ & h^{\alpha_1} \sumk \Resb \Ltwonorm{ \nabla \vh}{\el} 
+ h^{\alpha_2} \sumk \BResb \Ltwonorm{ \nabla \vh}{\el}  \leq C,
\end{align}
where $C$ is a positive constant dependent on the initial condition $\uu_0$. 
\end{corollary}
\begin{proof}
The first two terms of \eqref{Eq-T-3.1-part3-BV} are the same as \eqref{Eq::temp-corr}. The remaining terms can be obtained from Lemmas  \ref{Lem-app-ineq} and \ref{Lem-app-bnd-ineq}  by choosing $\gamma = \alpha_1$ and $\gamma = \alpha_2$  in part $(ii)$ of Lemmas \ref{Lem-app-ineq} and \ref{Lem-app-bnd-ineq}, respectively. The BV-estimate follows.
\end{proof}

Note that the spectral boundedness of the symmetrizer $\uu_\vv$ (and consequently $\vv_\uu$) as in \eqref{Eq::u_v_bound},  needs a deeper look.  In \cite{mythesis} it is shown that this seems achievable for some systems like shallow water equations and polytropic Euler equations by adopting some physical constraints
as well as $L_\infty$ bound on the approximate solution $\vh$. This is comparable to what Dutt \cite{dutt1988stable} established for the Navier-Stokes equations.

\section{Convergence analysis}
\label{sect::conv}
In the convergence analysis of SC-DG scheme \eqref{Eq-DGSC}, first in \S \ref{Sect:conv}  the convergence of the sequence of bounded solutions to a mv solution is proved. Then in \S \ref{Sect::cnsst} the admissibility of this solution is showed by satisfying some entropy inequality.

\subsection{Convergence to measure-valued solution}\label{Sect:conv}
In order to show convergence, we must  revisit and modify the proof given in  \cite{hiltebrand2014entropy} to account for the  removal of the streamline diffusion term. Furthermore, we employ refined estimates on several occasions, which leads to the less diffusive scaling of the shock-capturing operator  (cf. eq.~\eqref{Eq-D_SC}  in \S\ref{Sec:scop}).

First, let us introduce an $\Hone$-projection as the connection between infinite dimensional and finite dimensional space of the solution.
\begin{definition}\label{Def::H1proj}
The local  $\Hone$-projection of $\phiv \in \big(\Ccinf(\Omega \times \mathbb{R}_+)\big)^m$ into  $(\mathbb{P}_q)^m$ is denoted  by $\phivh$ and is defined as
$
\phivh = \proj_h(\phiv),
$
with $\proj_h \vert_{\el} \colon (\Hone(\el))^m \to (\mathbb{P}_q(\el))^m$;
where  for all $ \wh \in (\mathbb{P}_q(\el))^m$  we have
\begin{subequations} \label{Eq-H1prj}
\begin{align}
\int_{I_n} \! \int_K \! \inprod{\nabla \phiv^h}{\tilde{\uu}_\vv \nabla \wh} \dx \dt &=  \int_{I_n} \! \int_K \! \inprod{\nabla \phiv}{ \tilde{\uu}_\vv \nabla \wh} \dx \dt, \\
\int_{I_n} \!  \int_K \!  \phivh \dx \dt &= \int_{I_n} \! \int_K \! \phiv \dx \dt.
\end{align}
\end{subequations}
\end{definition}
Note that solving  \eqref{Eq-H1prj} corresponds to a discrete Neumann problem in $\el$. The regularity of the  solution of the elliptic problem and infinite
differentiability of $\phiv$ give the following estimates \cite{jaffre1995convergence}:
\begin{subequations}
\begin{align}
\Ltwonorm{\nabla \phivh}{\el} &\leq \Ltwonorm{\nabla \phiv}{\el}, \label{Eq-H1prjstab}\\ 
\Ltwonorm{ \phiv - \phivh}{\el} &\leq C  h^r \Ltwonorm{ \nabla^r \phiv}{\el}  \label{Eq-H1prjerr}, \\
\Ltwonorm{\phiv - \phivh}{\partial \el} &\leq C  h^{r-\frac{1}{2}} \Ltwonorm{\nabla^r \phiv }{\el}  \label{Eq-H1prjbndryerr},\\
\Ltwonorm{ \nabla (\phiv - \phivh)}{ \el} &\leq C  h^{r-1} \Ltwonorm{ \nabla^r \phiv}{\el}  \label{Eq-H1prjdrverr},
\end{align}
\end{subequations}
where $r=0,1, \ldots, q+1$. Note that \eqref{Eq-H1prjerr} and \eqref{Eq-H1prjbndryerr}) utilize $\Hm{2}$-regularity of the solution of the Neumann problem \eqref{Eq-H1prj}. This requires the convexity of the  triangulation $\mathcal{T}_n = \{ \el \}$.

Also we need the following estimates between $\Ltwo$  and $\Linf$. If $\phiv \in \big(\Ccinf(\el)\big)^m$, then the following estimates hold
\begin{subequations}
\begin{align}
\Ltwonorm{ \phiv}{\el} &\leq C h^{\frac{d'}{2}}  \Linfnorm{ \phiv}{\el}, \label{Eq::direct1}\\
\Hnorm{\phiv}{\el} &\leq C h^{\frac{d'}{2}} \Winfnorm{ \phiv }{\el} \label{Eq::direct2}.
\end{align}
\end{subequations}

In the following we assume that $q \geq 1$ (For the case $q = 0 $ our scheme reduces to a standard finite volume scheme for  which  convergence analysis is 
presented in \cite{cockburn1994error}).
The following theory establishes the convergence to mv solution for scheme \eqref{Eq-DGSC}:
\begin{theorem} \label{Thm::conv} Let $\vh$ be the approximate solution of the system \eqref{Eq::CL} by the shock capturing DG scheme \eqref{Eq-DGSC}. Under the assumption of \eqref{Eq::D_bound}, \eqref{Eq::u_v_bound} and 
\begin{equation}\label{Eq-inf-bnd}
\Linfnorm{ \vh}{\domt} \leq C,
\end{equation}
the approximate solution converges to a measure-valued solution \eqref{Eq-mvsol} of the system of conservation laws \eqref{Eq::CL}.
\end{theorem}
\begin{proof}
Consider $\{\vh\}_{h > 0}$  as the sequence of approximate solutions generated by SC-DG scheme \eqref{Eq-DGSC}. We first show that as $h \to 0$ the approximate solution is consistent with weak solution \eqref{Eq::weaksol} in the following sense
\begin{equation}
\label{Eq::weakcnsst}
\lim_{h \to 0}\int_0^T \! \int_\Omega  \!  \inprod{ \uu(\vh)}{ \phiv_t} + \sum_{k=1}^d \inprod{\ff^k(\vh)}{ \phiv_{x_k}}  \dx \dt = 0, \qquad  \forall \phiv \in (\Ccinf( \dom  \times (0, T))^m.
\end{equation} 
The consistency \eqref{Eq::weakcnsst} combined with Theorem \ref{Thm-mvsol}, is the key to prove the weak-$*$ convergence to a measure-valued solution. 

Let us choose $\phiv \in (\Ccinf( \dom  \times (0, T))^m$ and $\phivh = \proj_h(\phiv) $ (as in Definition~\ref{Def::H1proj}) and define the \emph{internal} and \emph{boundary} parts of DG formulation as the following  
\begin{align} \label{Eq::internalDG} 
\BDG^{(int)}(\vh, \phivh) &= -\! \sumk \int_{I_n} \! \int_{K} \!  \inprod{ \uu(\vh)}{\phivh_t}  +  \sum^d_{k=1} \inprod{ \ff^k(\vh)}{ \phivh_{x_k}}  \dx \dt, \\
\label{Eq::boundaryDG} 
\BDG^{(bnd)}(\vh, \phivh) &= \sumk \int_{I_n} \! \int_{\partial K } \!   \inprod{\fl}{ \phivh_{K, -}} \dss \dt  \\ & \quad + \sumk \int_{K} \inprod{\uu(\vh_{n+1, -})}{ \phivh_{n+1, -}} - \inprod{\uu(\vh_{n, -})}{\phivh_{n, +}} \dx \nonumber .
\end{align}
Using \eqref{Eq::DG-quasilinear} and \eqref{Eq-DGSC} we note that 
\begin{align}\label{Eq::conv-decomp}
\mathcal{B}(\vh, \phivh  ) = \BDG^{(int)}(\vh, \phivh) + \BDG^{(bnd)}(\vh, \phivh) +  \BSC(\vh, \phivh).
\end{align}
To prove consistency we observe that
\begin{align}\label{Eq::temp_convgence}
&\int_0^T \! \int_\Omega  \!  \inprod{ \uu(\vh)}{ \phiv_t} + \sum_{k=1}^d \inprod{\ff^k(\vh)}{ \phiv_{x_k}}  \dx \dt  \\
& \qquad  = \sumk \int_{I_n} \int_K \Big( \inprod{ \uu(\vh)}{\phiv_t } + \sum_{k=1}^d \inprod{\ff^k(\vh)}{ \phiv_{x_k}} \Big) \dx \dt \nonumber \\
& \qquad = \BDG^{(int)}(\vh, \phivh-\phiv) - \BDG^{(int)}(\vh, \phivh) \nonumber  \\
& \qquad  = \BDG^{(int)}(\vh, \phivh-\phiv) - \mathcal{B}(\vh, \phivh  ) + \BDG^{(bnd)}(\vh, \phivh  ) + \BSC(\vh, \phivh) \nonumber ,
\end{align} 
and seek to prove that \eqref{Eq::temp_convgence} $\to 0$ as $h \to 0$ as \eqref{Eq::weakcnsst}. Recall that $\mathcal{B}(\vh, \phivh) \equiv 0$ by definition of SC-DG scheme in \eqref{Eq-DGSC} and we refer to \cite{hiltebrand2014entropy} for the proof of the limit of $\break \BDG^{(int)}(\vh, \phivh-\phiv)$ and $\BDG^{(bnd)}(\vh, \phivh  )$. Here we only discuss the last term in \eqref{Eq::temp_convgence}.

We decompose the shock capturing term \eqref{Eq-B_SC} as follows:
\begin{align} 
\BSC(\vh, \phivh) &= \BSC^{(1)}(\vh, \phivh) + \BSC^{(2)}(\vh, \phivh) \nonumber \\
& =\sumk \int_{I_n} \int_K ( \epsk^{(1)} + \epsk^{(2)}) \Big(\inprod{\phivh_t}{ \tilde{\uu}_\vv \vh_t} + \sum_{k=1}^d \inprod{\phivh_{x_k}}{ \tilde{\uu}_\vv \vh_{x_k}}  \Big) \dx \dt,
\end{align}
where $ \epsk^{(1)}$ and $ \epsk^{(2)}$ correspond to cell and boundary residual parts in the viscosity coefficient in \eqref{Eq-D_SC}, respectively.
First, considering $\BSC^{(1)}$  and using  \eqref{Eq::u_v_bound} yields
\begin{align}
 \epsk^{(1)}
\leq C \frac{h^{\alpha_1} \Resb}{ \Ltwonorm{ \nabla \vh}{\el}}. 
\end{align}
Therefore, by using \eqref{Eq::direct2} and Cauchy-Schwarz inequality we get to
\begin{align}\label{Eq-T--B_SC1_bnd}
\vert \BSC^{(1)} (\vh, \phivh)\vert  \leq C  h^{\alpha_1}  \sumk \Resb \Hnorm{ \phivh}{\el}  
  \leq C h^{\tfrac{d'}{2} + \alpha_1}   \sumk \Resb. 
\end{align}
Using  Lemma \ref{Lem-app-ineq}, we find that  $\BSC^{(1)}(\vh, \phivh)$ 
vanishes as $h\to 0$, if  $\alpha_1 >0$. 

Similarly for $\BSC^{(2)}$ we obtain
\begin{equation*}
\vert \BSC^{(2)} (\vh, \phivh)\vert   \leq C h^{\tfrac{d'}{2} + \alpha_2}  \sumk \BResb.
\end{equation*}
Using Lemma \ref{Lem-app-bnd-ineq} and choosing $\alpha_2 > 0$ one can see that $\BSC^{(2)} (\vh, \phivh)$ vanishes as $h \to 0$. 

Owing to the $\Linf$ bound on $\vh$ as \eqref{Eq-inf-bnd} and based on the result of Theorem \ref{Thm-mvsol} , we can claim that there is a Young measure $\bm{\mu}$, such that
\begin{align}
 \uu(\vh) &\overset{*}{\rightharpoonup} \mv{ \bm{\mu}_y}{\uu(\bm{\sigma})},  \label{Eq::u_mv-conv} \\ 
\ff^k(\vh) &\overset{*}{\rightharpoonup} \mv{\bm{\mu}_y}{ \ff^k(\bm{\sigma})}. \label{Eq::f_mv-conv}
\end{align} 
as $h \to 0$. In other words unlike weak solutions, nonlinearity in $\uu(\vv)$  or $\ff^k(\vv)$ commutes with this new sense of convergence. This establishes the convergence we look for;
by \eqref{Eq::temp_convgence}, \eqref{Eq::u_mv-conv} and \eqref{Eq::f_mv-conv} we obtain
\begin{align}
\lim_{h \to 0}&\int_0^T \! \int_\dom \!  \inprod{\uu(\vh)}{\phiv_t} + \sum_{k=1}^d \inprod{\ff^k(\vh)}{ \phiv_{x_k}}  \dx \dt \nonumber \\
=  &\int_0^T \! \int_\dom \! \inprod{ \mv{\uu(\bm{\sigma})}{\bm{\mu}_y}}{ \phiv_t} + \sum_{k=1}^d \inprod{ \mv{\ff^k(\bm{\sigma})}{ \bm{\mu}_y}}{ \phiv_{x_k}}  \dx \dt =0,
\end{align}
and the Theorem \ref{Thm::conv} follows. \qquad
\end{proof}

\subsection{Entropy consistency} \label{Sect::cnsst}
The remaining step is showing that the solution obtained by \eqref{Eq-DGSC} is admissible, i.e.\ satisfies \eqref{Eq-mv-entrpy}. 
Before stating the corresponding theorem we introduce the following \emph{super approximation} estimate or \emph{discrete commutator property}:
\begin{lemma}
Let $\vh \in \findim$  and $\varphi$ is an infinitely smooth function $\varphi \in\nobreak \mathcal{C}^\infty (\el)$. Then the following results hold
\begin{align}
\Ltwonorm{ \varphi \vh - \proj_h(\varphi \vh)}{\el} &\leq C(\varphi) h \Ltwonorm{ \vh}{\el} \label{Eq::super_region} \\
\Ltwonorm{ \varphi \vh - \proj_h(\varphi \vh)}{\partial \el} &\leq C(\varphi) h^{1/2} \Ltwonorm{ \vh}{\el} \label{Eq::super_bnd}
\end{align}  
\end{lemma} 
The proof of \eqref{Eq::super_region} is a special case of the proof presented in \cite{bertoluzza1999discrete} and the boundary estimate \eqref{Eq::super_bnd} can be proved along the same line as \eqref{Eq::super_region}.

The entropy consistency result is given as the following theorem:
\begin{theorem} \label{Thm::entropy_cnsst}  
Let $\vh$ be the approximate solution generated by the scheme \eqref{Eq-DGSC}. We assume that $\vh$ is uniformly bounded as in \eqref{Eq-inf-bnd} and  the conditions \eqref{Eq::D_bound} and \eqref{Eq::u_v_bound} hold. Then, the limit measure-valued solution $\bm{\mu}$ satisfies the entropy
condition \eqref{Eq-mv-entrpy}.
\end{theorem}
\begin{proof}
We follow \cite{hiltebrand2014entropy} and consider an infinitely smooth non-negative  function $\allowbreak 0 \leq \nobreak \varphi \in\nobreak \Ccinf (\dom \times (0, T))$. 
Also in order that $\vh \varphi$ can be inserted as the test function in quasi-linear form $\mathcal{B}$, it needs to be projected to the finite dimensional space $\findim$. This is done using the $\Hone$-projection operator \eqref{Eq-H1prj} and results in the following two terms 
\begin{equation}\label{Eq::cnsst_projdecom}
\mathcal{B}(\vh, \proj_h(\vh \varphi)) =  \mathcal{B}(\vh, \vh \varphi) + \mathcal{B}(\vh, \proj_h(\vh \varphi) - \vh \varphi).
\end{equation}
As we will show, the second term, which is called \emph{compensation term},  vanishes as $h$ goes to zero while the first one provides us with the entropy inequality condition \eqref{Eq-mv-entrpy}. 

The first term can be decomposed in naive DG and shock capturing parts as
\begin{equation}\label{Eq::cnsst-decom}
\mathcal{B}(\vh, \vh\varphi) = \BDG(\vh, \vh \varphi) + \BSC(\vh, \vh\varphi).
\end{equation}
Along the same lines as in \cite{hiltebrand2014entropy}, one can prove that 
\begin{align}\label{Eq::cnsst-BDG}
\BDG(\vh, \vh\varphi) \geq -  \int^T_0 \! \int_\dom U(\vh) \varphi_t + \sum_{k=1}^d F^k(\vh) \varphi_{x_k} \dx \dt.
\end{align}
We do not repeat the proof here and  refer to \cite{hiltebrand2014entropy} for details. 
The shock capturing part, using  \eqref{Eq-B_SC}, can be written as 
\begin{align}\label{T-cnsst-Eq-B_SC}
\BSC(\vh, \vh\varphi) 
=&\sumk \int_{I_{n}} \int_{K} \epsk \Big(\inprod{ \vh_t}{ \tilde{\uu}_\vv \vh_t } + \sum^d_{k=1} \inprod{ \vh_{x_k}}{ \tilde{\uu}_\vv \vh_{x_k}} \Big) \varphi \dx \dt \nonumber \\ 
 &+ \underbrace{\sumk \int_{I_{n}} \int_{K} \epsk \Big(\inprod{ \vh \varphi_t}{ \tilde{\uu}_\vv \vh_t } + \sum^d_{k=1} \inprod{ \vh \varphi_{x_k}}{ \tilde{\uu}_\vv \vh_{x_k} } \Big) \dx \dt}_{A} \geq A,
\end{align}
and similar to \eqref{Eq-T--B_SC1_bnd} one can deduce 
that $ \vert A\vert \to 0 \text{ as } h \to 0 $. From \eqref{Eq::cnsst-decom}, \eqref{Eq::cnsst-BDG} and \eqref{T-cnsst-Eq-B_SC} we have
\begin{align*}
\mathcal{B}(\vh, \vh\varphi) \geq -  \int^T_0 \! \int_\dom  U(\vh) \varphi_t + \sum_{k=1}^d F^k(\vh) \varphi_{x_k}  \dx \dt  +	 A. 
\end{align*}
As $h \to 0$, $A$ vanishes and remembering the arguments on weak-$*$ convergence in Theorem~\ref{Thm::conv} yields
\begin{equation}\label{Eq::entropypart1}
\lim_{h \to 0} \mathcal{B}(\vh, \vh\varphi)) \geq - \int^T_0 \! \int_\dom \varphi_t \mv{U(\bm{\sigma})}{\bm{\mu}_y} + \sum^d_{k=1} 
\varphi_{x_k} \mv{F^k(\bm{\sigma})}{\bm{\mu}_y}   \dx \dt. 
\end{equation}

Now we deal with the compensation term in \eqref{Eq::cnsst_projdecom} which contains the projection error, 
\begin{equation*}
\cnsserr := \err{\vh \varphi} = \vh \varphi - \proj_h(\vh \varphi),
\end{equation*}
and in the following we  show that the compensation term vanishes as $h$ goes to zero: 
\begin{align} 
\label{Eq::compenserr}
\lim_{h \to 0} \mathcal{B}(\vh, \cnsserr) = 0. 
\end{align}

One can decompose $\mathcal{B}(\vh, \cnsserr)$ as  follows
\begin{align}
\label{Eq::entcnsst_decomp}
\mathcal{B}(\vh, \cnsserr) = \BDG^{(\Res)}(\vh, \cnsserr) + \BDG^{(rem)}(\vh, \cnsserr) + \BSC(\vh, \cnsserr),
\end{align}
with the following definitions
\begin{align}
\label{Eq::entcnsst:Res}
\BDG^{(\Res)}(\vh, \cnsserr) &= \ds \sumk \int_{I_n} \int_K \inprod{ Res}{ \err{\vh \varphi}} \dx \dt, \\
\label{Eq::entcnsst:rem}
\BDG^{(rem)}(\vh, \cnsserr)  &= \sumk \int_K \inprod{ \uu(\vh_{n, +}) - \uu(\vh_{n, -})}{ (\cnsserr)_{n, +} } \dx \\ 
& \quad - \sumk \int_{I_n} \int_{\partial K} \!\inprod{ \fl - \ff(\vh_{K, -})\cdot\mbf{n}}{
 (\cnsserr)_{K, -}}  \dss \dt \nonumber.
\end{align}
Now we need to show that each term in \eqref{Eq::entcnsst_decomp} vanishes as $h \to 0$.

First, the definition of $\Hone$-projection \eqref{Eq-H1prj} obviously yields 
\begin{equation}\label{Eq::SCcompenserr}
\BSC(\vh, \cnsserr ) = 0.
\end{equation} 
Using the definition of $\BDG^{(\Res)}$ as \eqref{Eq::entcnsst:Res} combined with \eqref{Eq::super_region} and  \eqref{Eq::direct1} gives
\begin{align}\label{Eq::Rescompenserr}
\vert \BDG^{(\Res)}(\vh, \cnsserr)  \vert & \leq C
\sumk \Resb \Ltwonorm{ \vh \varphi - \proj_h(\vh \varphi)}{\el} \nonumber \\
& \leq  C h^{1+ \frac{d'}{2}} \Linfnorm{ \vh}{\domt} \sumk  \Resb,  
\end{align}
which vanishes as $h$ goes to zero if $\alpha_1 < 2$. This comes from Lemma \ref{Lem-app-ineq} with $\gamma = 1 + \frac{d'}{2}$.

Moreover for $\BDG^{(rem)}$, by definition  \eqref{Eq-BRes} and estimates \eqref{Eq::super_bnd} and \eqref{Eq::direct1}, we obtain
\begin{align}
\label{Eq::remcompenserr}
\vert \BDG^{(rem)}(\vh, \cnsserr) \vert &\leq C \sumk \BResb \Ltwonorm{ \vh \varphi - \proj_h(\vh \varphi)}{\partial \el} \nonumber \\ 
& \leq Ch^{1/2}  \Linfnorm{ \vh }{\domt} \Big( h^\frac{d'}{2} \sumk  \BResb \Big).
\end{align}

Using Lemma \ref{Lem-app-bnd-ineq}, we observe that $\BDG^{(rem)}(\vh, \cnsserr)$ vanishes as $h$ goes to zero. 

Combining \eqref{Eq::SCcompenserr}, \eqref{Eq::Rescompenserr} and \eqref{Eq::remcompenserr} one can show \eqref{Eq::compenserr}. Then using \eqref{Eq::cnsst_projdecom}, \eqref{Eq::entropypart1} and \eqref{Eq::compenserr} yields
\begin{align*}
0 \equiv \lim_{h \to 0 } \mathcal{B}(\vh, \proj(\vh \varphi)) \geq - \int^T_0 \! \int_\dom \varphi_t \mv{U(\bm{\sigma})}{\bm{\mu}_y} + \sum^d_{k=1} 
\varphi_{x_k} \mv{F^k(\bm{\sigma})}{\bm{\mu}_y}   \dx \dt. 
\end{align*} 
This proves the entropy consistency introduced in \eqref{Eq-mv-entrpy}. \qquad
\end{proof}

\section{Numerical experiments}
\label{sec::numeric}
In this section we present some numerical experiments. First, in section~\ref{sec:wave}, we solve  a linear system of the one dimensional wave equation to show that the order of  convergence is optimal,  and the presence of the SC term does not ruin it. In this case (and in general for linear symmetrizable systems) the solution converges to the unique entropy solution. For more discussion on this claim we refer to Theorem 4.7 in \cite{hiltebrand2014entropy}.  
As examples of more general systems, we present in section~\ref{sec:shalloweq} a numerical solution of the dam break test case for the shallow water equations, as well as solutions of the  one dimensional Sod and Lax shock tube for the Euler equations in section~\ref{sec:eulereq}.

It is worth mentioning that the goal of presenting these results is to show that our proposed scheme can give acceptable results in practice. This section is not meant to verify the  analytical claim of convergence to emv solutions of one dimensional Euler or shallow water equations, respectively. The numerical proof of convergence to emv solution should be considered in some statistical approach, see \cite{kaman2012uncertainty} (and references therein) or \cite{fjordholm2014construction}.

The Netgen/Ngsolve library \cite{schoberl1997netgen} has been used for geometry handling and mesh generation as well as quadrature rules and the evaluation of basis functions. The nonlinear system obtained from the implicit space-time scheme is solved using damped Newton method utilizing the ILU preconditioned GMRES available through the PETSc library \cite{balay2014petsc}. 

There are some free parameters in the scheme which need to be selected including $C^{1,2}_{SC}, \alpha_{1, 2}$ and $\theta$. Unless otherwise mentioned explicitly we set them as $C^{1}_{SC} = 1$ and $\theta = 0.5$. Since our analytical results indicate that any  $\alpha_2>0$ can be chosen, we have set $C^{2}_{SC} = 0$. The value of $\alpha_{1}$ is set to $1.5$ in the case of wave equation and $1.3$ in the rest to be more diffusive. 
These settings give us acceptable result in most cases.

It should be noted that in the presented figures of the solution we draw the original solution polynomial elementwise without any additional limitation.

\subsection{Wave equation}
\label{sec:wave}
The wave equation in one dimension can be written as the following form
\begin{align}
h_t + c u_x &= 0, \\
u_t + c h_x &= 0,
\end{align}
where $c$ is some constant value. In this case the system is linear and symmetric in its original form and by choosing the entropy function as $U(\mbf{u}) = \frac{1}{2} (h^2 + u^2)$ the entropy variables would be the same as  the conservative variables.

Hence, the entropy conservative flux would be the simple average of the flux values at the edge and the diffusion operator is set to Rusanov type.
In our numerical test cases the boundary conditions are set to Dirichlet, the wave speed to $c = 1$ and the final time to $T = 1$. Also the calculation domain is considered as ${[0, 3]}$.
We use two different initial settings:
\subsubsection{Wave equation: smooth initial data}
We consider 
\begin{equation}
\label{Eq::wave_init_sm}
h(x, 0) = \sin(2 \pi x), \qquad u(x, 0) = \sin(2 \pi x) /3. 
\end{equation}
We solve this for polynomial degrees $q = 0, 1, 2, 3$  with and without shock capturing. The results are presented in 
Tables \ref{Tab::wave_smooth_no} and \ref{Tab::wave_smooth_SC}. One can observe that the naive DG formulation (i.e without any stabilization) is sufficiently good in this smooth case. Adding the shock capturing term merely adds some diffusive behaviour (in terms of slightly larger error reported in Table \ref{Tab::wave_smooth_SC}), while it does not affect the accuracy of the scheme in terms of rate-of-convergence. Asymptotically, we get the optimal order $q+1$ in convergence of the error in $L_1$ norm, even in the presence of the SC term. 

Only for very coarse meshes and high polynomial degree we see a significant contamination of the accuracy (cf. last column of Table \ref{Tab::wave_smooth_SC} 
). With refining the mesh, however, the order of convergence  will be  the order of consistency of the scheme, and the error levels are not significantly compromised by the SC term. 
\begin{table}[H]
\centering 
\caption{Convergence result for wave equation, smooth initial data without SC}
\label{Tab::wave_smooth_no}
\begin{tabular}{|c | *2c | *2c | *2c | *2c | }
\hline
  &  \multicolumn{2}{c}{$q=0$}  & \multicolumn{2}{c}{$q=1$} & \multicolumn{2}{c}{$q=2$} & \multicolumn{2}{c}{$q=3$} \vline \\
  \hline
h & $\Vert e \Vert_{L_1}$ & order & $\Vert e \Vert_{L_1}$ & order & $\Vert e \Vert_{L_1}$ & order & $\Vert e \Vert_{L_1}$ & order \\ 
  \hline
$\frac{1}{10}$ & 1.869  &&  4.668e-2  && 8.073e-3 && 1.997e-4 & \\
$\frac{1}{20}$ & 1.597   & 0.226 &  2.941e-2& 0.666 & 2.114e-3 & 1.933 & 1.079e-4  & 0.888 \\ 
$\frac{1}{40}$ & 1.146   & 0.477 & 6.328e-3 & 2.217 & 3.306e-4 & 2.677 & 6.994e-6 & 3.947\\
$\frac{1}{80}$ & 7.146e-1   & 0.682 & 1.410e-3 & 2.165 & 3.758e-5 & 3.137 & 3.788e-7 & 4.206\\
$\frac{1}{160}$ & 4.044e-1 & 0.821   & 3.243e-4 & 2.121 & 4.344e-6 & 3.113 & 2.099e-8 & 4.173\\
\hline
\end{tabular}
\end{table}

\begin{table}[H]
\centering 
\caption{Convergence result for wave equation, smooth initial data with SC}
\label{Tab::wave_smooth_SC}
\begin{tabular}{|c | *2c | *2c | *2c | *2c | }
\hline
  &  \multicolumn{2}{c}{$q=0$}  & \multicolumn{2}{c}{$q=1$} & \multicolumn{2}{c}{$q=2$} & \multicolumn{2}{c}{$q=3$}  \vline \\
  \hline
h & $\Vert e \Vert_{L_1}$ & order & $\Vert e \Vert_{L_1}$ & order & $\Vert e \Vert_{L_1}$ & order & $\Vert e \Vert_{L_1}$ & order \\ 
  \hline
$\frac{1}{10}$ & 1.869  &&  2.538e-1  && 2.570e-1 && 1.952e-1 & \\
$\frac{1}{20}$ & 1.597   & 0.226 &  5.862e-2 & 2.114 & 1.780e-2 & 3.852 & 1.483e-2  & 3.718 \\ 
$\frac{1}{40}$ & 1.146   & 0.477 & 9.690e-3 & 2.596 & 6.467e-4 & 4.782 & 1.784e-5 & 9.698\\
$\frac{1}{80}$ & 7.146e-1   & 0.682 & 1.748e-3 & 2.470 & 4.997e-5 & 3.694 & 6.120e-7 & 4.865\\
$\frac{1}{160}$ & 4.044e-1 & 0.821   & 3.537e-4 & 2.305 & 4.754e-6 & 3.393 & 2.562e-8 & 4.578\\
\hline
\end{tabular}
\end{table}

\subsubsection{Wave equation: discontinuous initial data}
We consider 
\begin{equation}
\label{Eq::wave_init_di}
(h, u)|_{t=0} = \begin{cases}
(1,\frac{1}{3}), \quad x < 1.5, \\ (0, 0), \quad x > 1.5.
\end{cases}
\end{equation}
Like the smooth case we solve the problem with different polynomial degree, and both with and without shock capturing. As we expect, based on Tables \ref{Tab::wave_disc_no} and \ref{Tab::wave_disc_SC} 
,  due to presence of a discontinuity, the order of convergence cannot be better than $1$, while increasing $q$ results in lower error.

Again, as in the smooth case, the presence of shock capturing mechanism does not affect the order of convergence. On the other hand, as Figure \ref{Fig::wave_disc_SC01} shows, while the naive DG implementation shows lots of oscillations in the solution in the vicinity of discontinuities, using the SC term reduces those oscillations considerably. Moreover, comparing the result of linear and quadratic elements shows that using higher order polynomials significantly helps in controlling the overshoot.

The effect of shock capturing parameters $\alpha_1$ and $\theta$ is shown in Figure \ref{Fig::wave_disc_alftth}. We observe that by increasing $\alpha_1$ and decreasing $\theta$ the solution become less diffusive. This shows that  the extension of the admissible range for $\alpha_1$ in our work compared to \cite{hiltebrand2014entropy} can make the method less diffusive. 

\begin{table}[H]
\centering 
\caption{Convergence result for wave equation, discontinuous initial data without SC}
\label{Tab::wave_disc_no}
\begin{tabular}{|c | *2c | *2c | *2c | *2c | }
\hline
  &  \multicolumn{2}{c}{$q=0$}  & \multicolumn{2}{c}{$q=1$} & \multicolumn{2}{c}{$q=2$} & \multicolumn{2}{c}{$q=3$} \vline \\
  \hline
h & $\Vert e \Vert_{L_1}$ & order & $\Vert e \Vert_{L_1}$ & order & $\Vert e \Vert_{L_1}$ & order & $\Vert e \Vert_{L_1}$ & order \\ 
  \hline
$\frac{1}{10}$ & 4.279e-1&		&  		1.211e-1  && 6.241e-2 & & 4.021e-2 & \\
$\frac{1}{20}$ & 3.435e-1& 	0.317 &  7.444e-2  & 0.702 & 4.311e-2 & 0.534 & 3.425e-2 & 0.231 \\
$\frac{1}{40}$ & 2.554e-1&	0.427&  4.466e-2  &0.737& 2.464e-2 & 0.807 & 1.908e-2  & 0.844 \\
$\frac{1}{80}$ & 1.835e-1  & 0.477 &  2.674e-2 &0.740 & 1.429e-2 & 0.786 & 1.089e-2 & 0.809\\
$\frac{1}{160}$ & 1.304e-1  &	0.493 &  1.569e-2  &0.769 & 8.168e-3 & 0.807& 6.157e-3 & 0.823\\
\hline
\end{tabular}
\end{table}

\begin{table}[H]
\centering 
\caption{Convergence result for wave equation, discontinuous initial data with SC}
\label{Tab::wave_disc_SC}
\begin{tabular}{|c | *2c | *2c | *2c | *2c | }
\hline
  &  \multicolumn{2}{c}{$q=0$}  & \multicolumn{2}{c}{$q=1$} & \multicolumn{2}{c}{$q=2$} & \multicolumn{2}{c}{$q=3$} \vline  \\
  \hline
h & $\Vert e \Vert_{L_1}$ & order & $\Vert e \Vert_{L_1}$ & order & $\Vert e \Vert_{L_1}$ & order & $\Vert e \Vert_{L_1}$ & order \\ 
  \hline
$\frac{1}{10}$ & 4.279e-1&		&  		1.207e-1  && 8.421e-2 & & 7.577e-2 & \\
$\frac{1}{20}$ & 3.435e-1& 	0.317 &  7.891e-2  & 0.613 & 5.518e-2 & 0.609 & 4.495e-2 & 0.753 \\
$\frac{1}{40}$ & 2.554e-1&	0.427&  4.628e-2  &0.769& 3.003e-2 & 0.878 & 2.305e-2  & 0.963 \\
$\frac{1}{80}$ & 1.835e-1  & 0.477 &  2.732e-2 &0.761 & 1.659e-2 & 0.856 & 1.252e-2 & 0.882\\
$\frac{1}{160}$ & 1.304e-1  & 0.493 &  1.586e-2 &0.784 & 9.371e-3 & 0.824 & 6.813e-3 & 0.876\\
\hline
\end{tabular}
\end{table}

\begin{figure}[htbp]
    \centering
    \begin{subfigure}[t]{0.49\textwidth}
        \centering
        \includegraphics[width=\textwidth]{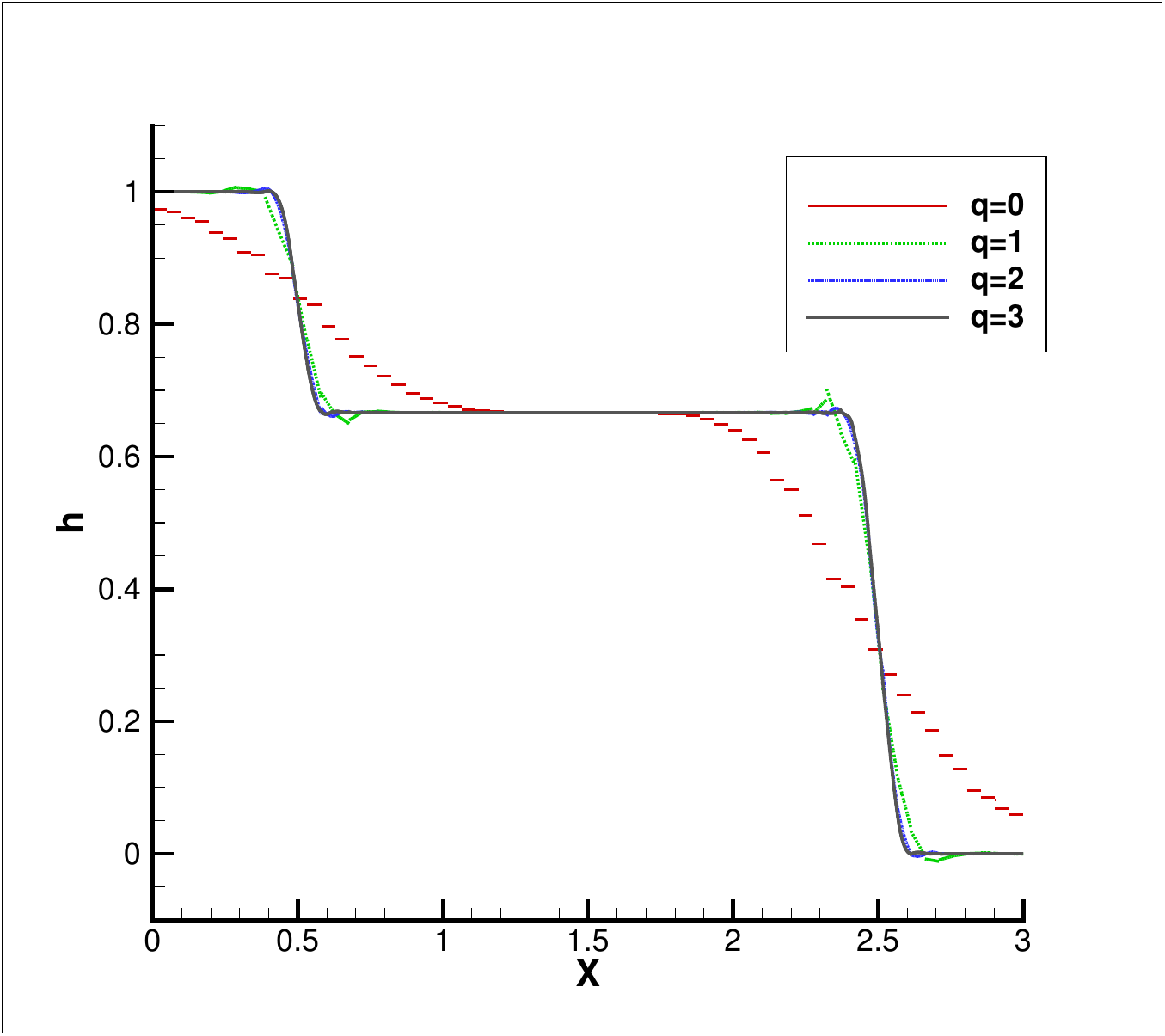}
        \caption{Different $q$, with SC}
    \end{subfigure}%
    ~ 
    \begin{subfigure}[t]{0.49\textwidth}
        \centering
        \includegraphics[width=\textwidth]{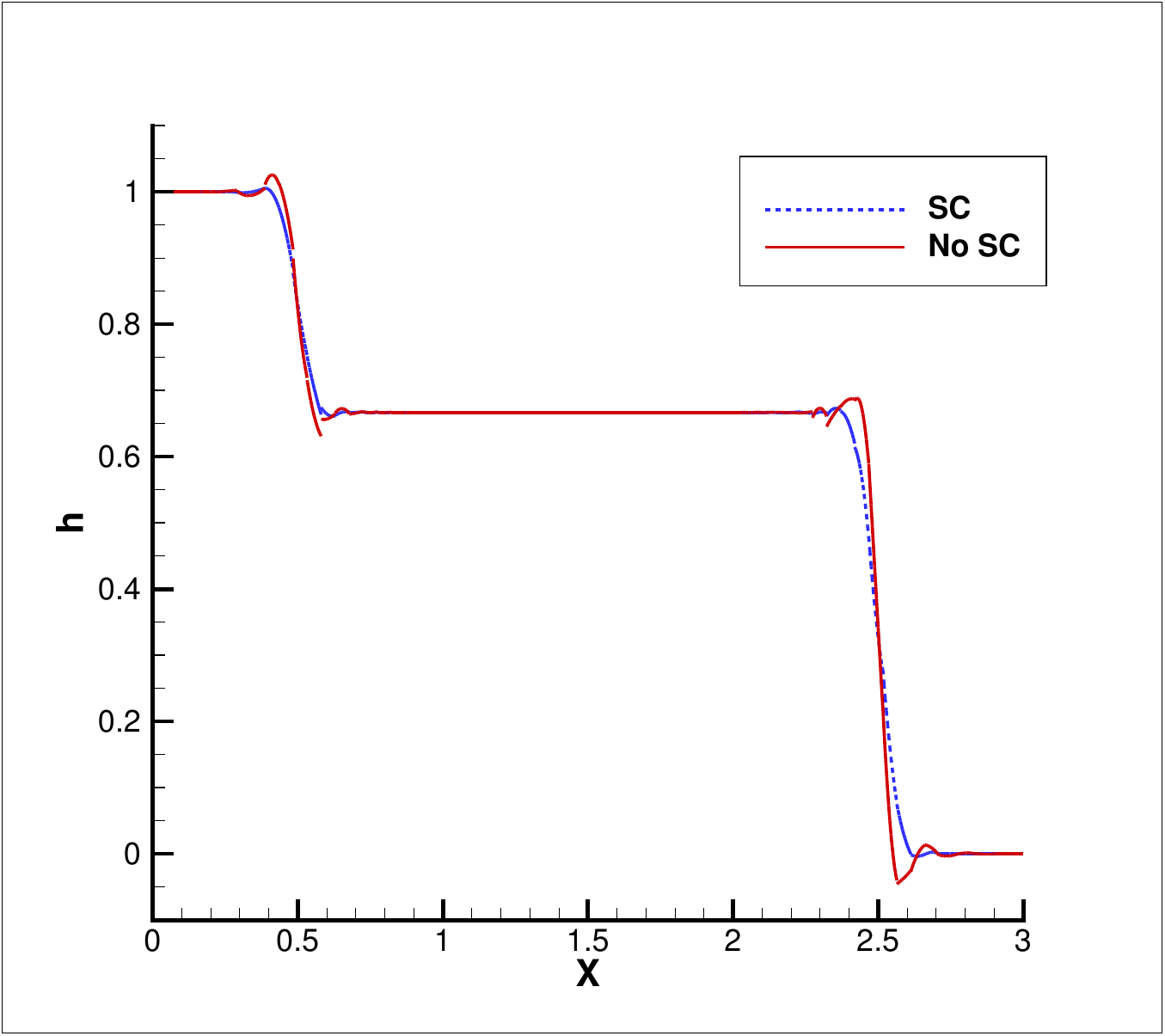}
        \caption{$q=2$, with/without SC}
    \end{subfigure}
     \caption{Wave equation, discontinuous initial data,  $h = 1/20$}
     \label{Fig::wave_disc_SC01}
\end{figure}

\begin{figure}[htbp]
    \centering
    \begin{subfigure}[t]{0.49\textwidth}
        \centering
        \includegraphics[width=\textwidth]{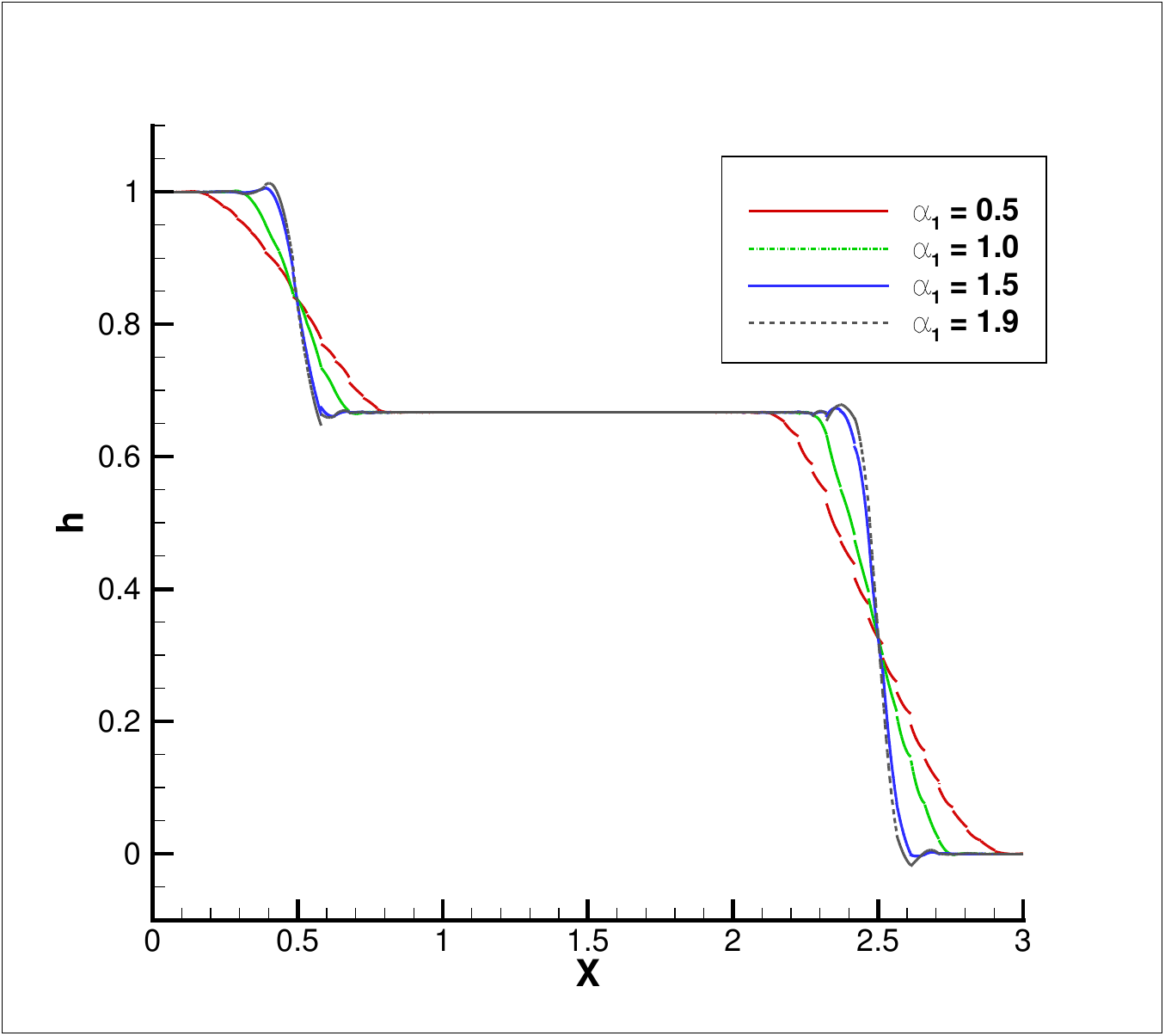}
    \end{subfigure}%
    ~ 
    \begin{subfigure}[t]{0.49\textwidth}
        \centering
        \includegraphics[width=\textwidth]{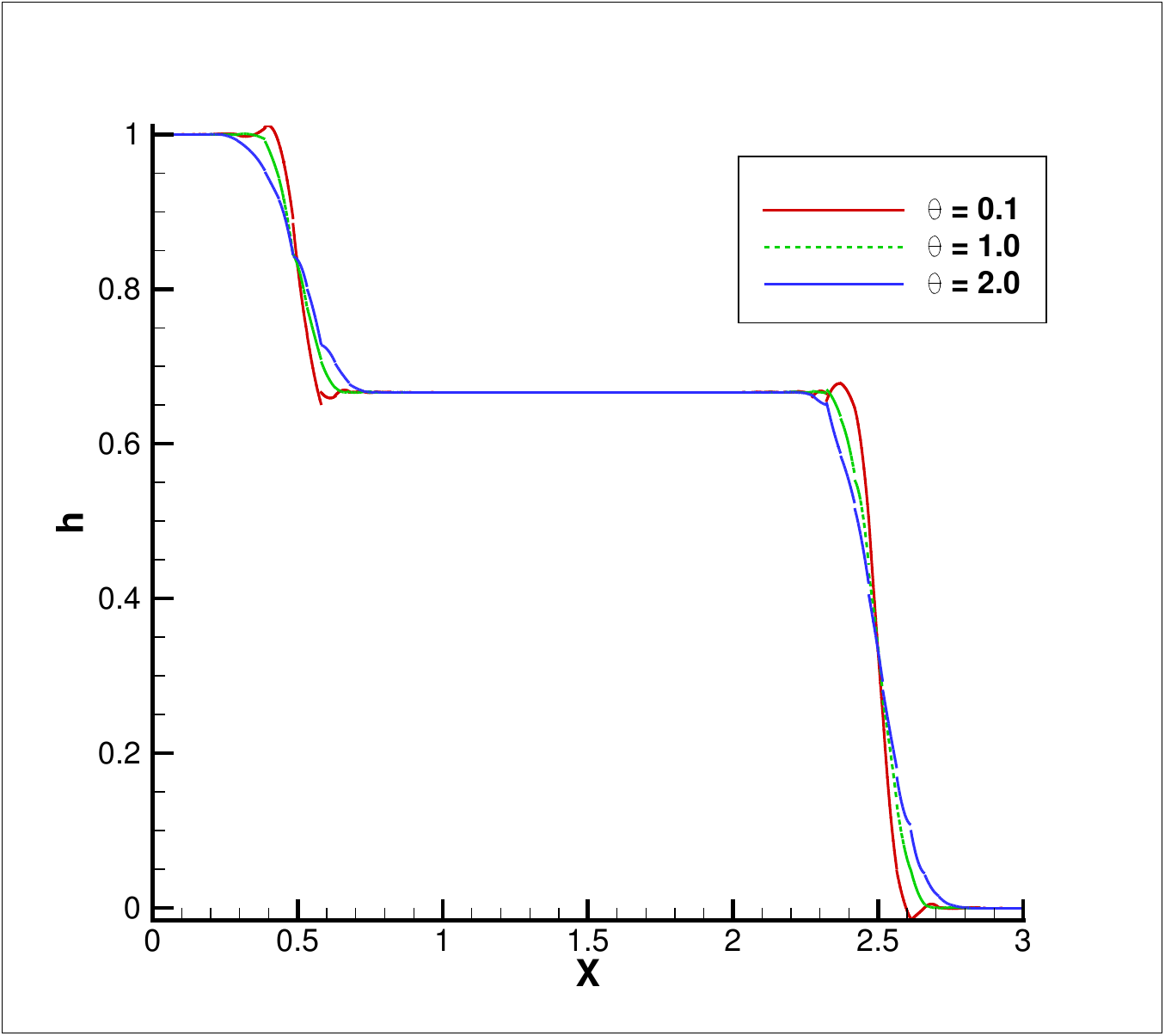}
    \end{subfigure}
     \caption{Wave equation, discontinuous initial data, $q=2$, $h = 1/20$, effect of two parameters, $\alpha_1$ (left) and $\theta$ (right)}
		\label{Fig::wave_disc_alftth}
\end{figure}
 
\subsection{Shallow water equations}
\label{sec:shalloweq}
The   shallow water equations which describe the disturbance propagation in incompressible fluids under the influence of gravity can be written as 
\begin{align}
h_t + (hu)_x &= 0, \\
(hu)_t + (h u^2 + \dfrac{1}{2} g h^2)_x	 &= 0,
\end{align}
where $h$ and $u$ are the depth and the velocity of the water, respectively and $g = 1$ is the gravity acceleration. The entropy function in this case is  defined as the total energy $U = \frac{1}{2} \left( h u^2+ g h^2\right)$. Hence, the corresponding entropy variables and entropy conservative flux can be set as in \cite{fjordholm2012arbitrarily}. Also we choose Rusanov type for the diffusion operator of the entropy stable flux. 

Moreover, we set the initial condition for \emph{dam break} problem as follows
\begin{align}
\label{Eq::wave_sw}
(h, u)|_{t=0} = \begin{cases}
(1.5,0), \quad &x < 0, \\ (1, 0), \quad &x > 0.
\end{cases}
\end{align}

We take the computational domain as ${[0, 10]}$, with Dirichlet boundary condition, and the final time is set to $T = 1$. In Figure \ref{Fig::sw}, we present the result with $q=2$ with/without shock capturing versus the exact solution calculated by SWASHES code \cite{FLD:FLD3741}.

The result and their comparison with the exact solution shows a good control of the shock with acceptable overshoot, and the shock is quite sharp.

\begin{figure}[htbp]
    \centering
    \begin{subfigure}[t]{0.49\textwidth}
        \centering
        \includegraphics[width=\textwidth]{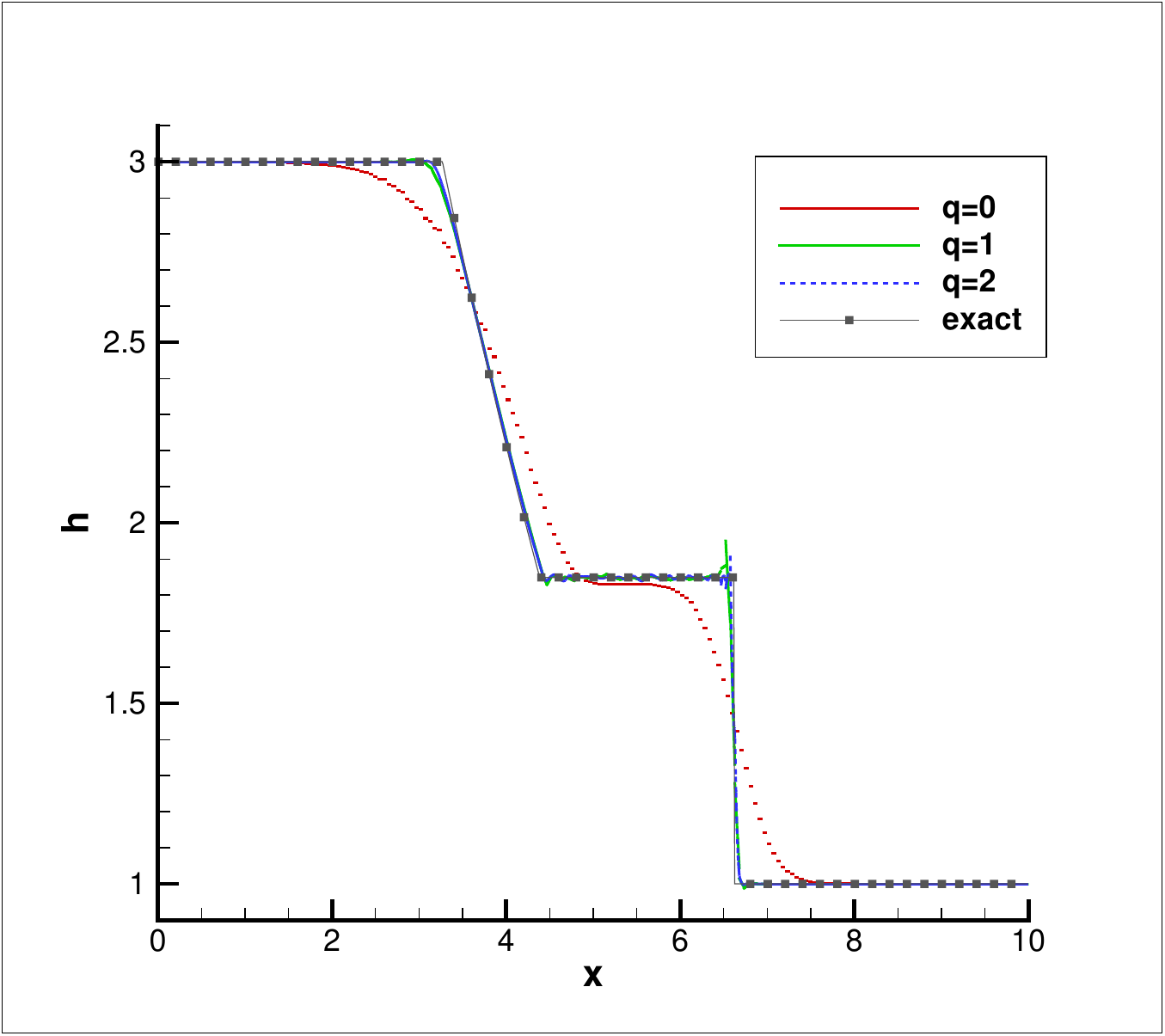}
        \caption{Different $q$, with SC}
    \end{subfigure}%
    \begin{subfigure}[t]{0.49\textwidth}
        \centering
        \includegraphics[width=\textwidth]{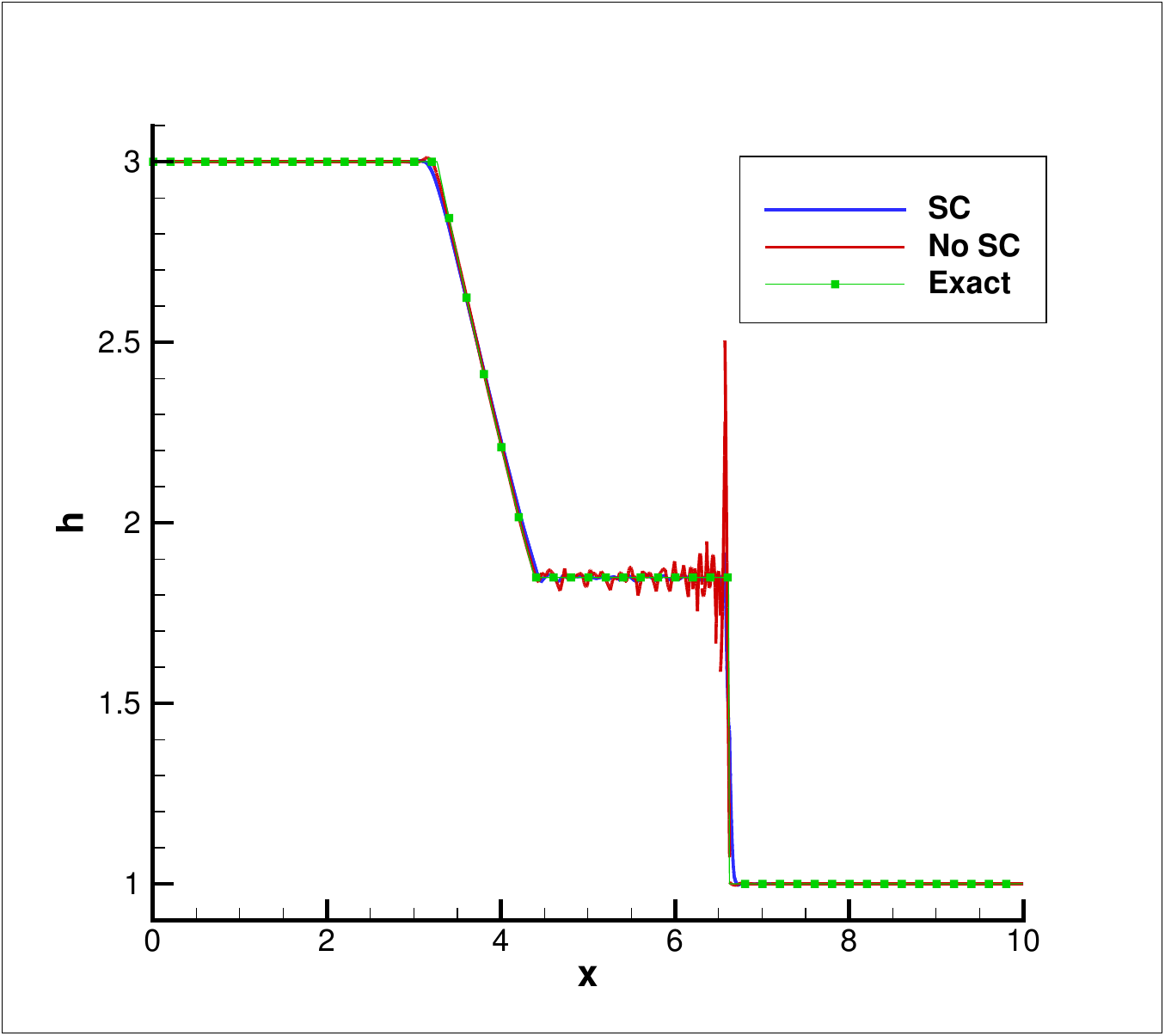}
        \caption{$q=2$, with/without SC}
    \end{subfigure}
     \caption{Dam break, height, $ h = 1/20$}
     \label{Fig::sw}
\end{figure}

\subsection{Euler Equations for polytropic gas}
\label{sec:eulereq}
The one-dimensional Euler equations  can be written as
\begin{align}
\rho_t + (\rho u)_x &= 0, \\
(\rho u)_t + (\rho u^2 + p)_x &= 0, \\
E_t + (u(E+p))_x &= 0,
\end{align}
where $\rho$, $ u $ and $E$ correspond to density, velocity and total energy of the gas, respectively. Here  $p$ is the pressure of the gas and is defined as $p = (\gamma -1) (E -\frac{1}{2 }\rho u^2)$, where $\gamma$ is the adiabatic exponent which is set to $1.4$ in all experiments here.

Following \cite{ismail2009affordable} the entropy function defined as 
$
U(\mbf{u}) = -\dfrac{\rho s}{\gamma -1 },
$
where $s$ is the specific entropy defined as $s = \ln p - \gamma \ln \rho $. The corresponding definition of entropy variables and entropy conservative flux $\mbf{f}^\star$ as well as diffusion operator is defined according to  \cite{ismail2009affordable}.
We consider two types of Riemann problems for our numerical test in the domain ${[0, 10]}$. The boundary conditions are set to Dirichlet type with the following initial conditions
\begin{equation}\label{Eq::RiemannInit}
(\rho, u, p)_{t=0} = \begin{cases}
(\rho_L, u_L, p_L) \qquad x < 5, \\ (\rho_R, u_R, p_R) \qquad x \geq 5,
\end{cases}	
\end{equation}
which we define as the right and left states for the following two cases:
\subsubsection{Sod shock tube}
Here the initial condition is in the form \eqref{Eq::RiemannInit} with the values
\begin{equation}
(\rho_L, u_L, p_L) = (1, 0, 1), \qquad (\rho_R, u_R, p_R) = (0.125, 0, 0.1). 
\end{equation}
The results are presented in Figure \ref{Fig::sod}. We observe that the presented shock capturing mechanism acts effectively near both shock waves and the contact discontinuity. Our solution compares well  to the results of \cite{hiltebrand2014entropy}, which are improved by some pressure scaling as well as streamline diffusion. Moreover while increasing polynomial degree from $q=0$ to $q=1$ significantly improves the solution quality, the quadratic polynomial solution is quite similar to the linear one, and only improves the overshoots near the shock wave. 

\begin{figure}[htbp]
    \centering
    \begin{subfigure}[t]{0.49\textwidth}
        \centering
        \includegraphics[width=\textwidth]{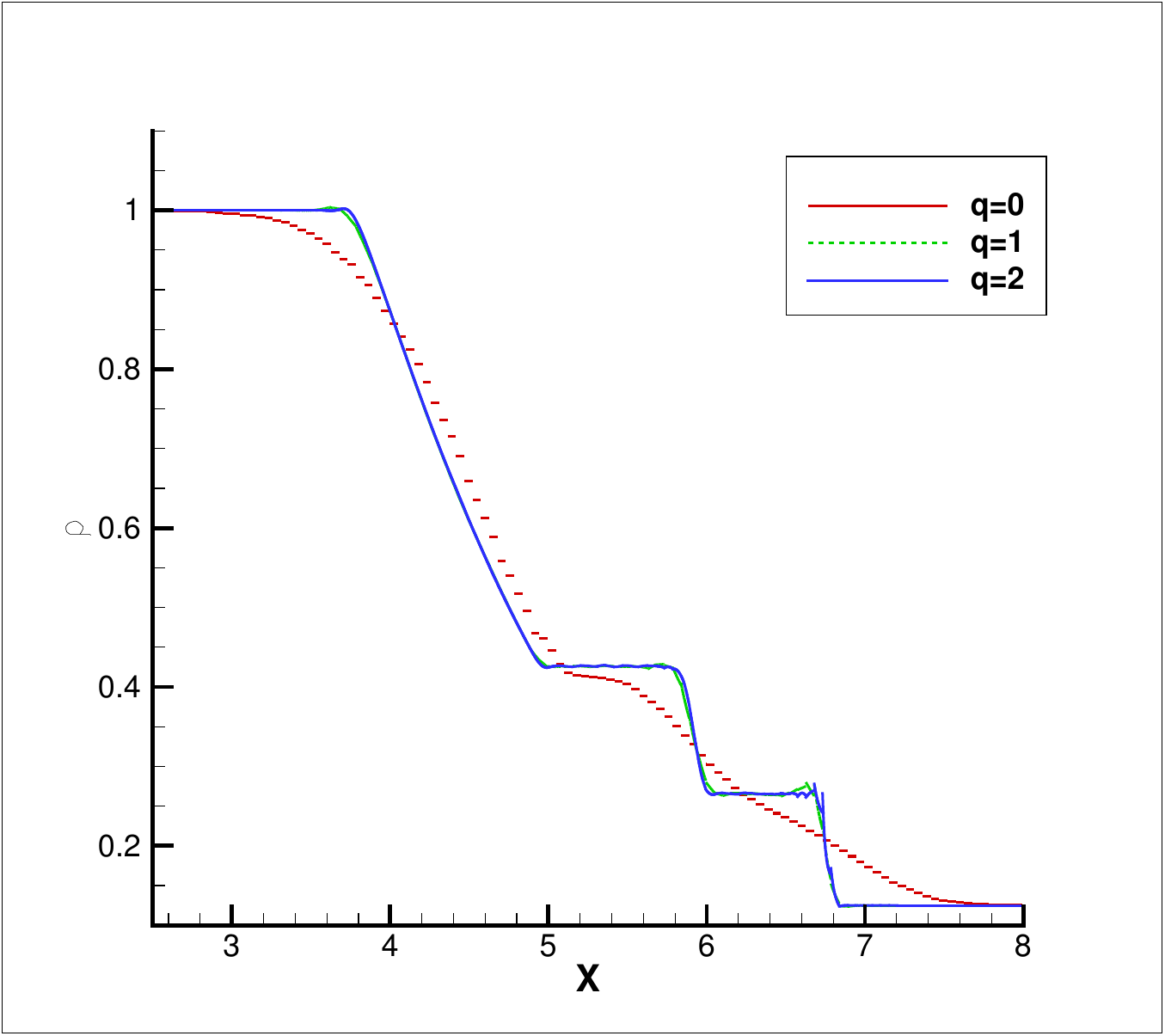}
       \caption{Different $q$, with SC}
    \end{subfigure}%
    ~ 
    \begin{subfigure}[t]{0.49\textwidth}
        \centering
        \includegraphics[width=\textwidth]{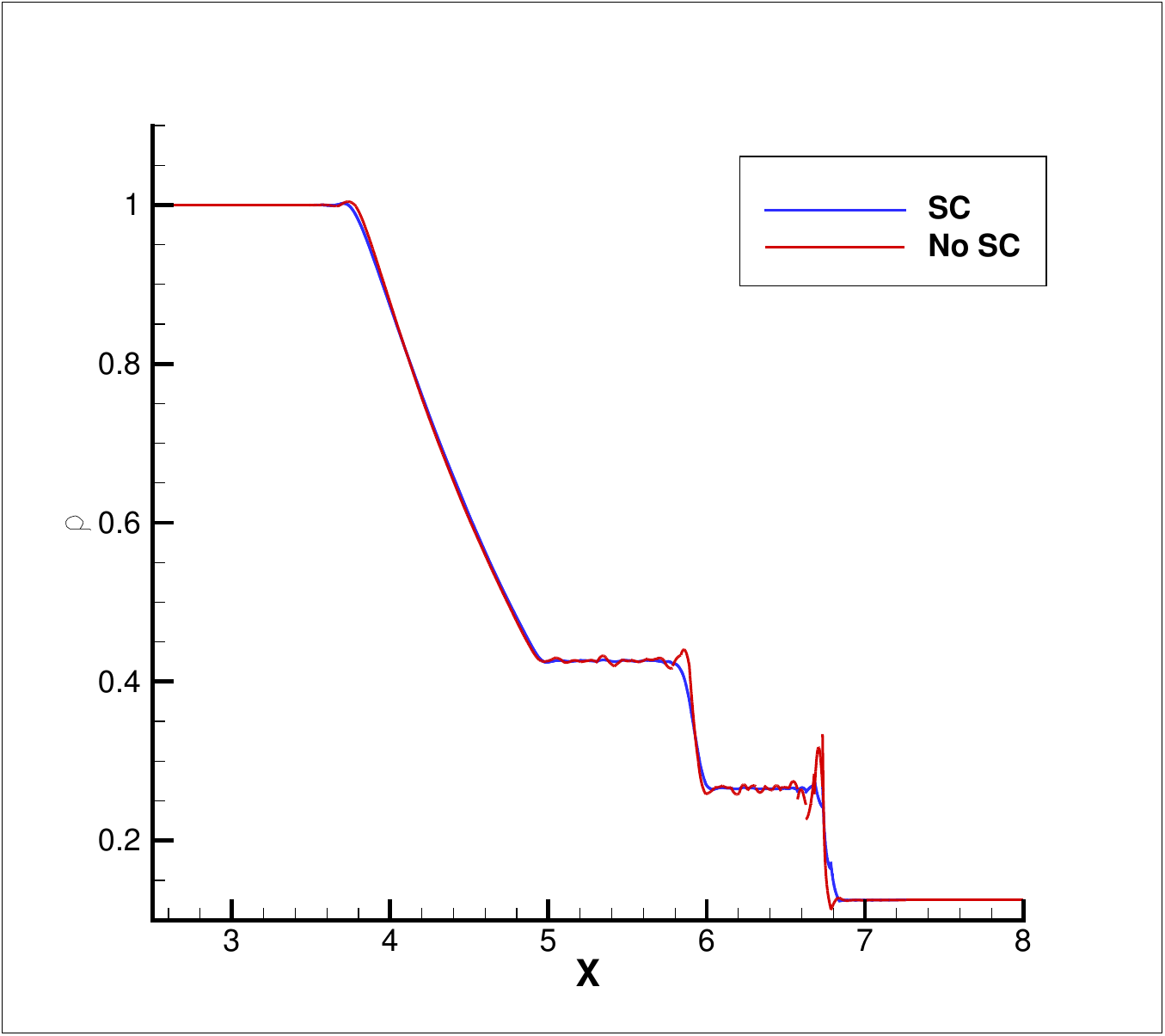}
         \caption{$q=2$, with/without SC}
    \end{subfigure}
    \caption{Sod shock tube, density, $ h = 1/20$}
	\label{Fig::sod}
\end{figure}

\subsubsection{Lax shock tube}
Here the initial condition is in the form \eqref{Eq::RiemannInit} with values
\begin{equation}
(\rho_L, u_L, p_L) = (0.445, 0.698, 3.528), \qquad (\rho_R, u_R, p_R) = (0.5, 0, 0.571). 
\end{equation}
Again, comparing results with \cite{hiltebrand2014entropy} shows that the shock capturing mechanism is effective in alleviating the oscillations. The general behaviour here is similar to the Sod case, but with larger overshoots due to the stronger shock.
\begin{figure}[htbp]
    \centering
    \begin{subfigure}[t]{0.49\textwidth}
        \centering
        \includegraphics[width=\textwidth]{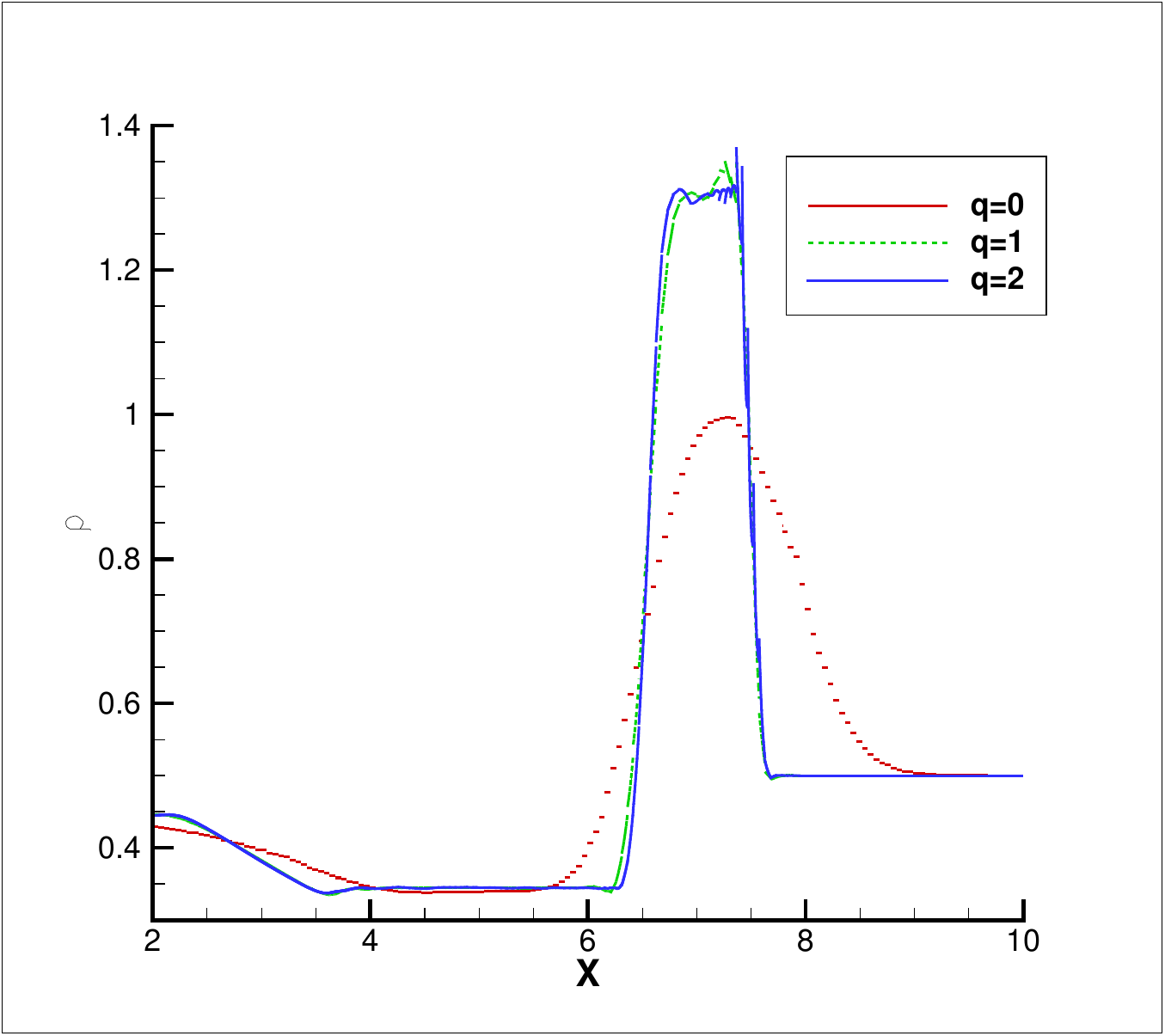}
        \caption{Different $q$, with SC}
    \end{subfigure}%
    ~ 
    \begin{subfigure}[t]{0.49\textwidth}
        \centering
        \includegraphics[width=\textwidth]{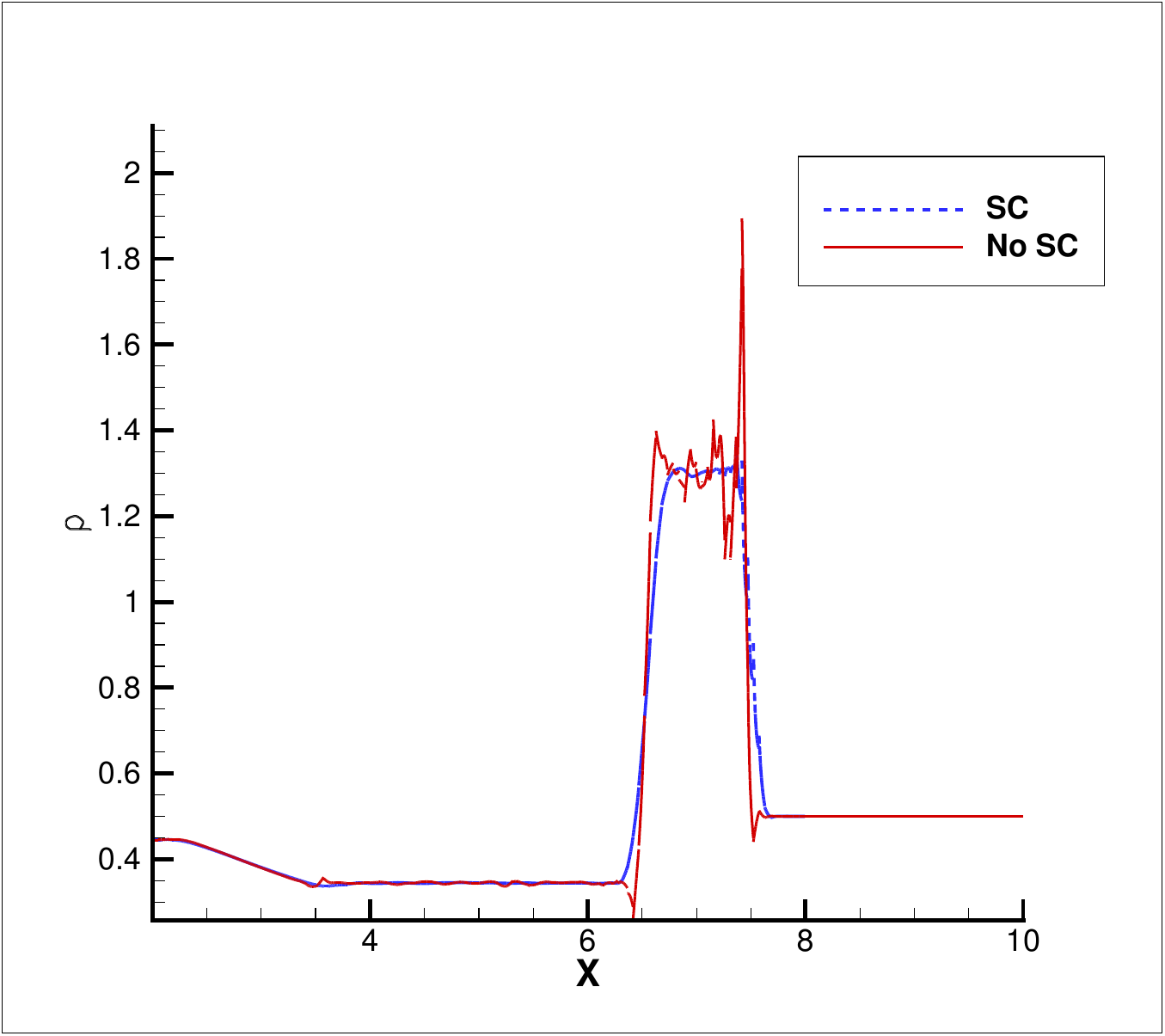}
        \caption{$q=2$, with/without SC}
    \end{subfigure}
     \caption{Lax shock tube, density, $ h = 1/20$ }
	 \label{Fig::lax}
\end{figure}

\section{Conclusion}
\label{Sec::conc}
In this work we have shown the capability of the shock capturing mechanism to ensure the convergence to entropy measure-valued solution for nonlinear systems of  conservation laws. We followed the framework presented in \cite{hiltebrand2014entropy} for streamline diffusion shock capturing discontinuous Galerkin methods, and introduced a stripped-down version by omitting the streamline diffusion term while retaining the entropy stability and convergence of the method. Also using super approximation estimates, we succeeded to `relax' the scaling in the viscosity and obtain a less diffusive method. Furthermore, the applicability of the method was presented through numerical experiments.

An improved version of our scheme might consider a dimensionally consistent  of the shock capturing operator. (See \cite{hiltebrand2014efficient} for a dimensionally consistent formulation with SD term.) This is left for future work.

\section*{Acknowledgments}
The authors thank Dr. Andreas Hiltebrand from ETHZ for his useful comments on \cite{mythesis} which is the base of this paper. Moreover the first author acknowledges Micheal Woopen for help on the numerical code.
The authors are supported by the Deutsche Forschungsgemeinschaft (German Research
Association) through grant GSC 111.

\bibliography{mybib.bib}{}
\bibliographystyle{siam}
\Appendix

\section{Proof of Lemma \ref{Lem-app-ineq} and Lemma \ref{Lem-app-bnd-ineq}}
\label{Sec::App}
Here we present the proofs of the Lemmas \ref{Lem-app-ineq} and \ref{Lem-app-bnd-ineq}.
Note that by notation $\Gamma$ we mean a $h$-dependent constant  $\Gamma = C h^\beta$, where $C$ is independent of $h$. 

Assuming that \eqref{Eq::u_v_bound} and \eqref{Eq::f_v_bound} hold and  remembering the definition  \eqref{Eq::residual}, the residual can be bounded from above as  $\vert \Res \vert \leq C \vert \nabla \vh \vert $.
Consequently  one can easily obtain 
\begin{equation}\label{Eq-app-resbarineq}
\Resb \leq  C  \Ltwonorm{ \nabla \vh}{\el}. 
\end{equation}

\subsection{Proof of Lemma \ref{Lem-app-ineq}}
\begin{enumerate}[(i)]
\item We split the summation into summations on $\kappa^>  := \{\kappa: \Ltwonorm{\nabla \vh}{\el} \geq \Gamma \}$ and $\kappa^< := \{\kappa: \Ltwonorm{ \nabla \vh}{\el} < \Gamma \}$ as
\begin{align}
h^\gamma \sumk \Resb = h^\gamma \left( 
\sum_{ \kappa \in \kappa^> } \!\! \Resb +  \sum_{\kappa \in \kappa^< } \!\! \Resb \right) = I_1 + I_2.
\end{align}
 
We estimate each of terms $I_1$ and $I_2$ separately:
\begin{itemize}
\item Bound on $I_1$: 
Remembering the definition of $\epsk$ in \eqref{Eq-D_SC} gives
\begin{align}
I_1 &\leq  C h^{\gamma-\alpha_1} \!\!\! \sum_{\kappa \in  \kappa ^>}  \!\epsk \Ltwonorm{ \nabla \vh}{\el}  + h^{\gamma-\alpha_1} \!\!\sum_{\kappa \in  \kappa^>} \! \epsk  h^\theta \nonumber \\
& \leq C \frac{h^{\gamma-\alpha_1}}{\Gamma} \sumk \epsk \Ltwonorm{ \nabla \vh}{\el}^2 + \frac{h^{\gamma-\alpha_1+\theta}}{\Gamma^2} \sumk \epsk \Ltwonorm{ \nabla \vh}{\el}^2  \nonumber \\
& \leq  C \left(\frac{h^{\gamma-\alpha_1}}{h^\beta} + \frac{h^{\gamma-\alpha_1+\theta}}{h^{2\beta}} \right),
\end{align}
where \eqref{Eq::temp-corr} is used in the last estimate. 
\item Bound on $I_2$: Using \eqref{Eq-app-resbarineq} one can show  that  $\Resb < C \Gamma $ holds where $\Ltwonorm{ \nabla \vh}{\el} \leq \Gamma$ and consequently 
\begin{equation}
I_2 < C h^{\gamma+\beta} \left( \sumk 1 \right)  \leq C h^{-d'} h^{\gamma+\beta},
\end{equation}
where the term $h^{-d'}$ stands for the number of all space-time elements in the domain which is true thanks to the quasi-uniformity condition \eqref{Eq-uniformity}.
\end{itemize}
Considering bounds on $I_1$ and $I_2$, yields
\begin{equation}\label{Eq::hienq1-final}
h^\gamma \sumk \Resb \leq C \left(h^{\gamma-\alpha_1 - \beta} + h^{\gamma-\alpha_1+\theta - 2\beta} +  h^{\gamma+\beta-d'} \right).
\end{equation}
For \eqref{Eq::hienq1-final} to be bounded (regarding to $h$) it is required that 
\begin{subequations}\label{Eq-hineq-1}
\begin{align}
\gamma-\alpha_1 - \beta &\geq 0, \label{Eq-hineq-1-1} \\
\gamma-\alpha_1 + \theta - 2 \beta &\geq 0, \label{Eq-hineq-1-2}\\
\gamma + \beta - d' &\geq 0. \label{Eq-hineq-1-3} 
\end{align}
\end{subequations}
If one can find a possible value for (here the only) free parameter $\beta$, then \eqref{Eq::hienq1-final}  is bounded by initial condition implied in $C$ and the diameter of the space-time domain. Also this bound goes to zero as $h \to 0$ in the case of strict inequality.

Considering \eqref{Eq-hineq-1-1} and \eqref{Eq-hineq-1-3} gives
\begin{equation} \label{Eq::temp-lem-app-1}
d' - \gamma \leq \beta \leq \gamma -\alpha_1,
\end{equation} 
which implies  $\gamma \geq \frac{d' + \alpha_1}{2}$. A similar calculation using  \eqref{Eq-hineq-1-2} and \eqref{Eq-hineq-1-3} leads to the condition $\gamma \geq \frac{2d' + \alpha_1 - \theta}{3}$. Using the condition on $\theta$ in \eqref{Eq::tetabound}, one can check that the second condition reduces to the first one, and we only need to satisfy $\gamma \geq \frac{d' + \alpha_1}{2}$.  This completes the proof of part (i) of Lemma \ref{Lem-app-ineq}. 

Note that the maximum rate of convergence with respect to $h$ occurs when  $\beta = \theta = \frac{d' - \alpha_1}{2}$. For this choice all terms in brackets on the right hand side of \eqref{Eq::hienq1-final}  reduce to  $h^{\gamma - \frac{\alpha_1 + d'}{2}}$.
\item We show that we can find a uniform upper bound in case of $\gamma = \alpha_1$. Then the theorem is obviously true for $\gamma > \alpha_1$. 

From the \eqref{Eq::visc_bound} and the definition of $\epsk$ in \eqref{Eq-D_SC} we have
\begin{equation}\label{Eq::Lemma-app2-a}
\sumk \frac{h^{\alpha_1} \Resb \Ltwonorm{\nabla \vh}{\el}^2 }{\Ltwonorm{\nabla \vh}{\el} +  h^\theta } \leq C.   
\end{equation}
Using the arguments of Lemma \ref{Lem-app-ineq} with $\gamma = \theta + \alpha_1$, we can claim that
\begin{equation}\label{Eq::Lemma-app2-b}
\sumk h^{\alpha_1} h^\theta \Resb \leq C,   
\end{equation}
if  $\theta + \alpha_1 \geq \frac{d' + \alpha_1}{2}$, i.e.  $\theta \geq \frac{d' - \alpha_1}{2}$. This is true by condition \eqref{Eq::tetabound}.

Now, one should note that 
\begin{equation}\label{Eq-Prop-Hilte}
\Ltwonorm{\nabla \vh}{\el} \leq \max \{h^\theta, \frac{2 \Ltwonorm{\nabla \vh}{\el}^2 }{\Ltwonorm{\nabla \vh}{\el} + h^\theta} \},
\end{equation}
which can be easily seen by a graphical argument. 
Using \eqref{Eq::Lemma-app2-a} and \eqref{Eq::Lemma-app2-b} combined with \eqref{Eq-Prop-Hilte} yields
\begin{equation*}	
h^{\alpha_1} \sumk \Resb \Ltwonorm{\nabla \vh}{\el} \leq h^{\alpha_1} \sumk \Resb \max \{h^\theta, \frac{2 \Ltwonorm{\nabla \vh}{\el} ^2 }{\Ltwonorm{\nabla \vh}{\el} + h^\theta} \} \leq C.
\end{equation*}

The bound $C$ vanishes as $h\to 0$ if $\gamma > \alpha_1$.
\end{enumerate}
 
\subsection{Proof of Lemma \ref{Lem-app-bnd-ineq}}

\begin{enumerate}[(i)]
\item 
Using~\eqref{Eq::temp-corr}, \eqref{Eq::D_bound} and the definition of $\BResb$ \eqref{Eq-BRes}, one can conclude that the first and last term of $\sum \BResb^2$  are bounded. 
For the second term, using the definition of the entropy conservative flux \eqref{Eq::entrp-cnsrv-flux} and its consistency yields 
\begin{align}\label{Eq::boundred_proof}
\flent - \ff(\vh_{K, -}) \cdot \mbf{n} &= \ff^\star (\vh_{K, -}, \vh_{K, +}; \mbf{n}) - \ff^\star (\vh_{K, -}, \vh_{K, -}; \mbf{n}) \nonumber \\ &= \int_0^1 \ff^k(\vh(\theta)) - \ff^k (\vh_{K,-})  \df{\theta}  \nonumber \\
&= \int^1_0 \theta \sum_{i=1}^m  a_i \ff^k_\vv (\mbf{b}_i(\theta)) \df{\theta} \, \jump{ \vh_K} 
\end{align} 
with coefficients $a_i \in [0, 1]$ such that $\sum
_{i=1}^m a_i =1$ and $\mbf{b}_i(\theta)$s are some values on the straight line connecting $\vh_{K, -}$ and $\vh(\theta)$. The value $\vh(\theta)$ is defined by the parameterization introduced in \eqref{Eq::barth-param}.  The last identity is the result of the mean value theorem for a vector-valued function (cf.\ e.g.\ \cite{mcleod1965mean}).

By assuming the boundedness as in \eqref{Eq::f_v_bound}, \eqref{Eq::boundred_proof} can be bounded from above by
\begin{equation}\label{Eq::Lip_assump}
\vert \flent - \ff(\vh_{K, -}) \cdot \mbf{n} \vert \leq C \vert \jump{\vh_K} \vert.
\end{equation}
The  estimate~\eqref{Eq::temp-corr}   combined with \eqref{Eq::Lip_assump} leads to
\begin{equation}
\sumk \BResb^2\leq C,
\end{equation}
and the proof completes with recalling that $\ds h^\frac{d'}{2} \sumk \BResb \leq C \big( \sumk \BResb^2 \big)^{1/2}$.

\item The proof   proceeds along the same lines as the proof of part (ii) of Lemma \ref{Lem-app-ineq}, by using \eqref{Eq::visc_bound} and  the uniform bound presented in
part (i). 

In the proof it is needed to have $\theta + \alpha_2 \geq \frac{d'}{2}$ which implies $\theta \geq \frac{d'}{2} - \alpha_2$. This is the requirement on the regularization parameter $\theta$ in \eqref{Eq::tetabound}. 
\end{enumerate}

\end{document}